\documentclass[a4paper,12pt]{amsart}
\usepackage{comment,amsmath,amssymb,amsthm,changepage,pifont,
            graphicx,tikz,xcolor,longtable,mathtools}
\usepackage[english]{babel}
\usetikzlibrary{patterns}
\usetikzlibrary{decorations.pathreplacing}

\usepackage{ytableau}
\usepackage{enumitem}
\allowdisplaybreaks

\newtheorem{maintheorem}{Theorem}

\newtheorem{theorem}{Theorem}[section]

\newtheorem{lemma}[theorem]{Lemma}
\newtheorem{corollary}[theorem]{Corollary}
\newtheorem{proposition}[theorem]{Proposition}

\theoremstyle{remark}
\newtheorem{remark}[theorem]{Remark}
\newtheorem{definition}[theorem]{Definition}

\newtheorem{example}[theorem]{Example}

\def\AA{{\mathbb A}}

\def\CC{{\mathbb C}}

\def\FF{{\mathbb F}}

\def\NN{{\mathbb N}}

\def\PP{{\mathbb P}}
\def\QQ{{\mathbb Q}}

\def\cI{{\mathcal I}}

\def\cO{{\mathcal O}}

\def\cR{{\mathcal R}}

\DeclareMathOperator{\aff}{aff}

\DeclareMathOperator{\mult}{mult}

\DeclareMathOperator{\Hom}{Hom}

\DeclareMathOperator{\Gr}{Gr}

\newcommand{\Naff}{N_\mathrm{aff}}

\usepackage[colorinlistoftodos]{todonotes}
\usepackage[colorlinks=true, allcolors=blue]{hyperref}

\title{A geometric determinant method and geometric dimension growth}

\author{Tijs Buggenhout}
\address{Department of Mathematics, KU Leuven, Belgium}
\email{tijs.buggenhout@kuleuven.be}

\author{Yotam I.\,Hendel}
\address{Department of Mathematics, Ben-Gurion University of the Negev, Be’er Sheva, Israel}
\email{yhendel@bgu.ac.il}

\author{Floris Vermeulen}
\address{Mathematics M\"unster, University of M\"unster, Germany}
\email{florisvermeulen.math@gmail.com}

\date{}

\begin{document}


\begin{abstract}
We study a geometric version of the dimension growth conjecture. While it is closely related in spirit to themes arising in geometric Manin's conjecture, it applies in greater generality and provides more uniform  bounds. 

For an irreducible projective variety $X$ defined over $\CC(t)$, the set $X(b)$ of $\CC(t)$-rational points on $X$  of degree less than $b$ has a natural structure of an algebraic variety over $\CC$. 
We study the dimension and irreducibility of $X(b)$ when $X$ has degree $d \ge 2$, and obtain  a geometric analogue of the classical dimension growth conjecture, namely that 
$\dim X(b) \le b\dim X  $ for every $b \ge 1$. 
 In particular, when $X$ is defined over $\CC$, this provides uniform bounds on the dimension of the space of degree $b$ rational curves  on $X$.

We also develop a 
geometric version of Heath-Brown's $p$-adic determinant method for varieties defined over $\CC(t)$.  
This allows us to show that as soon as $d \ge 6$, the number of irreducible components of $X(b)$ of dimension $b\dim X$  
is  bounded by a polynomial in $d$ which is independent of $b$. 
As a further application, we obtain an analogue of the Bombieri--Pila theorem for affine curves, as well as a corresponding result for projective curves.
\end{abstract}
\maketitle

\section{Introduction}

Since the work of Bombieri--Pila\,\cite{Bombieri-Pila}, the determinant method has been a crucial tool in obtaining uniform upper bounds on integral and rational points on algebraic varieties. Through a $p$-adic adaptation of this determinant method, Heath-Brown~\cite{Heath-Brown-Ann} was able to prove the \emph{dimension growth conjecture} for surfaces. This conjecture asserts more generally that an irreducible projective variety $X\subset \PP^n$ of degree $d\geq 2$ over $\QQ$ satisfies
\[
N(X,B)\ll_{n,d,\varepsilon} B^{\dim X + \varepsilon},
\]
where $N(X,B)$ counts the number of rational points of $X$ of height at most $B$. Note that the implicit constant in the inequality above depends only on $n, d$ and $\varepsilon$, and not on the variety $X$ itself. Using an affine variant of dimension growth, Browning--Heath-Brown--Salberger, and later Salberger were able to prove the dimension growth conjecture whenever $d\geq 4$ \cite{Brow-Heath-Salb,Salberger-dgc}, while the case $d=2$ is known via different methods \cite{Heath-Brown-Ann}. The case $d=3$ remains open, although a non-uniform version  has been obtained by Salberger \cite{Salberger-dgc}.

In recent years, there has been growing interest in making the implicit constant in dimension-growth type bounds explicit, and in particular in proving that it depends polynomially on $d$. 
In addition, it was shown that the $\varepsilon$ appearing in the exponent of the above bounds can be removed. 
Building upon work of many authors \cite{Walsh,CCDN-dgc, CDHNV, BCN-d}, it is now known thanks to Binyamini--Cluckers--Kato \cite{BinCluKat}, that if $X\subset \PP^n$ is irreducible of degree $d\geq 4$ and defined over $\QQ$, then
\[
N(X,B)\ll_n d^2 B^{\dim X}\log (B)^{O(1)}.
\]
This bound is essentially optimal, in view of the recent work~\cite{CluGla}. Let us also mention that nowadays dimension growth type bounds exist for all global fields, see \cite{Sedunova, CFL, Vermeulen:p, Pared-Sas, CDHNV, Vermeulen:affinedg}.

For many counting problems over global fields, it is possible to formulate geometric or motivic analogues. If $X$ is a projective variety over $\CC$, then for a positive integer $b$ the set 
\[
\Hom^b(\PP^1, X) = \{f: \PP^1\to X\mid \deg f = b\}
\]
can be considered a suitable analogue for the set of rational points of (logarithmic) height $b$. Note that this set has a structure of an algebraic variety over $\CC$, and hence the geometry of $\Hom^b(\PP^1, X)$ is naturally of interest. For example, one can try to understand the dimension, irreducibility or singularities of $\Hom^b(\PP^1, X)$. 
These types of problems received significant attention in recent years, notably in the case where $X$ is a smooth variety (see \cite{Kollar1996_RationalCurves,BrowningVishe, BiluBrowning, Faisant, Glas,faisant2025motivic}) 
or smooth with additional assumptions as in the context of geometric Manin's  conjecture (see e.g.~\cite{LeTa19_GeomManin,LeTa21_PrimeFano3folds} 
as well as 
\cite{Tanimoto_GeometricManin_Miyako2021} for a survey on geometric Manin's conjecture). 


Let $X\subset \PP^n$ be a quasi-projective variety over $\CC(t)$. Then similarly the set
\begin{equation}
\tag{$\star$} 
X(b) := \{(a_0 : \ldots : a_n)\in X(\CC(t))\mid a_i\in \CC[t] \text{ and } \deg a_i<b, ~\forall 0 \le i \le n\},
\end{equation}
which is the analogue of the set of rational points of  height smaller than $b$ on $X$, has a natural structure of a quasi-projective variety over $\CC$. It is an object closely related to, but different from $\Hom^b(\PP^1, X)$; if $X$ is defined over $\CC$ then $X(b)$ is the union of $\Hom^i(\PP^1, X)$ for $i<b$. 

The aim of this article is to investigate 
the geometry of $X(b)$ under minimal assumptions on $X$, and more precisely to develop and study a geometric analogue of the dimension growth conjecture in this setting. 
In contrast to the aforementioned works,  where $X$ is assumed to be smooth (and in many cases satisfies additional assumptions), here $X$ may be badly singular. In addition, instead of working with varieties over $\CC$, we work with varieties over $\CC(t)$.

While upper bounds on the dimension and number of irreducible components of $X(b)$ at this level of generality are less precise for a fixed $X$, we expect their uniformity to be useful in applications, as in the classical dimension growth setting (see e.g.~\cite{Walkowiak,Pared-Sas:Hilbert,GIP24}). 
 
\subsection{Main results} 
Our first main result is a geometric analogue of dimension growth bounds for $X(b)$. We note again our height normalization in the definition of $X(b)$ (e.g.~$X(0)=\varnothing$, see $(\star)$  above).

\begin{maintheorem}\label{thm:main.siegel}
Let $X\subset \PP^n$ be an irreducible projective variety of degree $d\geq 2$ and dimension $m$ defined over $\CC(t)$. Then for every $b \ge 1$,
\[
\dim X(b)\leq mb.
\]
\end{maintheorem}
Note that in the above theorem any degree $d\geq 2$ is allowed, including the case $d=3$ for which the classical dimension growth conjecture is still open. 
In addition, as shown in Example \ref{ex: attaining upper bound}, the above bound is sharp (note that 
$\dim X(b) \ge mb-1$  as soon as $X$ contains a linear space of dimension $m-1$ defined over $\CC$). 

One can ask for more geometric information about $X(b)$, such as its number of irreducible components. In fact, we believe that the natural analogue of dimension growth over $\QQ$ should include some control on the number of irreducible components of $X(b)$ of largest possible  dimension. Indeed, if a bound of the form
\[
N(X,B)\leq c B^m
\]
holds in the global field setting for some $c,m>0$, then the analogue over $\CC(t)$ should be that the dimension of $X(b)$ is bounded by $mb$, \emph{and} the number of irreducible components of $X(b)$ of dimension $mb$ is bounded by $c$. For such a result, we obtain the following. 

\begin{maintheorem}\label{thm:main.dgc}
Let $X\subset \PP^n$ be an irreducible projective variety of degree $d\geq 6$ and dimension $m$  defined over $\CC(t)$. Then for $b\geq 1$ we have
\[
\dim X(b)\leq mb,
\]
and the number of irreducible components of $X(b)$ of dimension $mb$ is bounded by $O(d^7)$. 
If $b\geq 7$ then this number of irreducible components is bounded by $O(d^4)$.\footnote{In both cases, the bound on the number of irreducible components is independent of $n$.} 
\end{maintheorem}

The proofs of Theorems~\ref{thm:main.siegel} and~\ref{thm:main.dgc} both proceed via showing the stronger affine analogues 
Theorems \ref{thm:aff.dim.growth} and \ref{thm:aff.dgc.Siegel}, which are proven by induction on $m=\dim X$. 
 For $X\subset \AA^n$ a quasi-affine variety over $\CC(t)$, we denote by
\[
X(b) = \{(a_1, \ldots, a_n)\in X(\CC(t))\mid a_i\in \CC[t] \text{ and } \deg a_i<b \text{ for each } 0 \le i \le n\}
\] 
the set of integral points of height at most $b$ on $X$.
Again, $X(b)$ is naturally a quasi-affine variety over $\CC$.
A suitable uniform bound on $\dim X(b)$, with $X\subset \AA^n$, yields both of the above theorems for projective varieties by passing to the affine cone. The reason for moving to the affine situation is that induction on dimension works much better, as the integral points of height at most $b$ in $\AA^n$ are contained in a $b$-dimensional family of hyperplanes over $\CC$. 

The key difference between the proofs of Theorem~\ref{thm:main.siegel} and~\ref{thm:main.dgc} is the base of the induction.

 For Theorem~\ref{thm:main.siegel} we follow the approach outlined in~\cite{CDHNV}, where the induction begins with $m=1$, i.e.\ the curve case. For this we prove 
a strong bound on 
the number of integral points of bounded height on plane affine curves with at least two points at infinity. 

\begin{proposition}\label{prop:curve.siegel}
Let $C\subset \AA^2$ be an irreducible curve over $\CC(t)$. If $C$ has at least two distinct points at infinity, then for every $b\geq 1$,
\[
\dim C(b)\leq 1.
\] 
\end{proposition}

If $C$ has genus $1$, or genus $0$ and at least $3$ points at infinity, Proposition~\ref{prop:curve.siegel} is closely related to Siegel's theorem on integral points over function fields (see Theorem \ref{thm:Mason}).
If $C$ has genus at least $2$, then this theorem follows from the Mordell conjecture over function fields (see Theorem \ref{thm:Manin-Grauert}).
Therefore the only remaining case is when $C$ is of genus $0$ with $2$ points at infinity, for which analyzing Pell equations over function fields plays a role. 
Since it seems difficult to control the number of irreducible components of $C(b)$, this result does not suffice to prove Theorem~\ref{thm:main.dgc}. 

The proof of Theorem \ref{thm:main.dgc} is more involved. We follow the approach of~\cite{BHB06}, where the base of the induction is for surfaces. 
To this aim, we develop a geometric version of the determinant method in the style of~\cite{BHB06} but with improvements from~\cite{CCDN-dgc, Walsh}. Over $\QQ$, the global determinant method of Salberger~\cite{Salb.upcoming} yields the strongest results around dimension growth, in particular leading to a proof of uniform dimension growth for projective varieties of degree $d\geq 4$. This global determinant method constructs a single auxiliary variety catching all rational or integral points on the variety of interest. Such a construction seems impossible in the geometric setting, since an affine curve defined over $\CC(t)$ can contain infinitely many integral points of bounded height. As such, it is not clear to us whether a geometric version of Salberger's global determinant method exists. This is the reason why we restrict to $d\geq 6$ in Theorem~\ref{thm:main.dgc}, while dimension growth over global fields is known for $d\geq 4$.

\begin{remark}
One can  approach Theorem~\ref{thm:main.siegel} and Theorem~\ref{thm:main.dgc} using spreading out methods\footnote{We thank Tim Browning for pointing this out to us.} as in~\cite{BrowningVishe}.
The key point is to have good control of the dependence on $q$ for dimension growth results over $\FF_q(t)$. 
For example, using~\cite[Thm.\,1.3]{Vermeulen:affinedg}, this approach leads to a quick proof of Theorem~\ref{thm:main.siegel} when $X$ has degree $d\geq 65$.
However, the more general results from~\cite{Pared-Sas, CDHNV, BinCluKat} for global fields leave the dependence on the field implicit, and while this dependence can likely be made explicit, this would require significant work.
Additionally, we note that spreading out  cannot be used to prove Theorem~\ref{thm:main.siegel} when $d=3$, simply because dimension growth over $\FF_q(t)$ remains unproven for $d=3$.
\end{remark}

In addition to proving Theorem \ref{thm:main.dgc}, we use the determinant method to deduce an analogue of the Bombieri--Pila theorem for affine curves~\cite{Bombieri-Pila}  and Heath--Brown's bounds for projective curves~\cite{Heath-Brown-Ann}. 
While a geometric version of the Bombieri--Pila bounds already appears in~\cite{CCL-PW}, where a bound on the  dimension of $C(b)$ for an affine curve $C$ is given, bounding the number of irreducible components is new. 
The geometric version of Heath--Brown's bound for projective curves is new in terms of both dimension and the number of irreducible components. 

\begin{maintheorem}[Affine curves]
\label{thm: affine curves}
Let $C\subset \AA^n$ be an irreducible affine curve of degree $d$ over $\CC(t)$. Then for every $b \ge 1$, 
\[
\dim C(b)\leq \left\lceil \frac{b}{d}\right\rceil.
\]
Moreover, the number of irreducible components of $C(b)$ of dimension $\lceil b/d\rceil$ is bounded by $O(d^3b)$.
\end{maintheorem}

\begin{maintheorem}[Projective curves]
\label{thm: projective curves}
Let $C\subset \PP^n$ be an irreducible projective curve of degree $d$ over $\CC(t)$. Then for every $b \ge 1$,
\[
\dim C(b)\leq \left\lfloor \frac{2(b-1)}{d}\right\rfloor +1 \le \left\lceil \frac{2b}{d}\right\rceil.
\]
Moreover, the number of irreducible components of $C(b)$ of dimension $\left\lfloor \frac{2(b-1)}{d}\right\rfloor +1$ is bounded by $O(d^3b)$. 
\end{maintheorem}

These results follow from the more precise Theorem~\ref{thm:proj.aux.poly} and Proposition~\ref{prop:aff.aux.poly}, which take into account the height of the defining polynomial of $C$.

For a projective surface one typically expects most rational points to accumulate on lines contained in the  surface.
Using the determinant method, we show that when removing all lines on our surface, one obtains strictly better bounds than in Theorem~\ref{thm:main.dgc}.  
This is similar to classical dimension growth results in the literature 
(see e.g.\,\cite[Theorem 10]{Heath-Brown-Ann} or \cite[Theorem 0.5]{Salberger-dgc}).

\begin{maintheorem}[Projective surfaces]\label{thm:proj.surfaces}
Let $X\subset \PP^3$ be an irreducible projective surface of degree $d\geq 6$ over $\CC(t)$.
Let $L$ denote the union of all lines $\ell \subset X$, 
and let $U\subset X$ be the  complement of $L$ in $X$, which is a quasi-projective variety. 
Then for every $b \ge 1$, 
\[
\dim U(b)\leq \left(1+\max\left\{ \frac{1}{2}, \frac{3}{2\sqrt{d}} + \frac{1}{3}\right\}\right) b + 2.
\]
\end{maintheorem}

Note that in the above result the quantity $\max\left\{ \frac{1}{2}, \frac{3}{2\sqrt{d}} + \frac{1}{3}\right\}$ is strictly smaller than $1$, since $d\geq 6$.
Hence this result improves upon the bound coming from Theorem~\ref{thm:main.dgc} (for $b \ge 5$).

\subsection*{Structure of the paper} 
In Section \ref{sec: prelims} we introduce needed definitions, and give some basic preliminary results such as a geometric version of the Schwartz--Zippel bounds and the Bombieri--Vaaler theorem. We also recall a version of Siegel's theorem over function fields (due to Mason), and a version of the Mordell conjectures over function fields (due to Manin--Grauert). 
Finally, we provide some examples which illustrate the geometric behavior of $X(b)$ in several interesting situations. 
In Section  \ref{sec: counting points on curves}, 
 we prove Proposition~\ref{prop:curve.siegel}, which  forms the base of the induction for Theorem~\ref{thm:main.siegel}. 
 In Section \ref{sec: the determinant method} we develop a geometric version of the determinant method needed to control the number of irreducible components of largest possible dimension, and to prove Theorem \ref{thm:main.dgc}. We also prove Theorems \ref{thm: affine curves} and \ref{thm: projective curves}.  
In Section \ref{sec:geom.dim.growth}, we prove Theorems~\ref{thm:main.siegel}, \ref{thm:main.dgc} and~\ref{thm:proj.surfaces}. 
The projective case is deduced from the affine versions 
Theorems~\ref{thm:aff.dgc.Siegel} and~\ref{thm:aff.dim.growth}, which we  prove by induction on dimension. The base case for this induction is counting on affine surfaces, which is given in Proposition \ref{prop:aff.surfaces}. 

\subsection*{Acknowledgements}  
The authors wish to thank Raf Cluckers for many useful discussions, for suggesting this problem, and for his encouragement and support.  
They also thank Tim Browning for useful discussions, and Lo\"{i}s Faisant for useful discussions and for comments on an earlier version of the paper.
T.\,B.\,was supported by FWO Flanders (Belgium) with grant number 1131925N. 
Y.\,I.\,H. was partially supported by FWO Flanders (Belgium) with grant number 12B4X24N.
F.\,V.\,was supported by FWO Flanders (Belgium) with grant number 11F1921N and by the Humboldt Foundation.

\section{Preliminaries}
\label{sec: prelims}

\subsection{Definitions and notation}

In general, throughout this paper we will denote $K = \CC(t)$ and $\cO_K = \CC[t]$. Given $x = (x_0 : \ldots : x_n)$ in $\PP^n(K)$, we may assume that the $x_i$ are in $\cO_K$ and coprime. Then we define the \emph{height} of $x$ to be
\[
h(x) = \max_i \deg x_i.
\]
Similarly, if $x = (x_1, \ldots, x_n)\in \cO_K^n$  we define the height of $x$ to be
\[
h(x) = \max_i \deg x_i.
\]
For $X\subset \PP^n_K$ a quasi-projective variety and $b$ a positive integer, we denote by 
\[
X(b) = \{x\in X(K): h(x) < b\}\subset \PP^n(K)
\]
the set of \emph{rational points of height less than $b$ on $X$}. For a quasi-affine variety $X\subset \AA^n_K$, we similarly define 
\[
X(b) = \{x\in X\cap \cO_K^n: h(x)<b\} \subset \cO_K^n
\]
to be the set of \emph{integral points of height less than $b$ on $X$}. 

\begin{lemma}\label{lem:Xb.is.variety}
Let $b$ be a positive integer. 
\begin{enumerate}
\item \label{it:aff.Xb} If $X\subset \AA^n_K$ is a quasi-affine variety, then $X(b)$ is naturally a quasi-affine variety\footnote{Note that $X(b)$ may not be reduced (see e.g.~Example \ref{ex: non-reduced C(b)}), in which case we can consider the associated reduced variety. For our purposes, this makes little difference.} 
 in $\AA^{bn}_{\CC}$.
\item \label{it:proj.Xb} If $X\subset \PP^n_K$ is a quasi-projective variety, then $X(b)$ is naturally a quasi-projective variety in $\PP^{b(n+1)-1}_{\CC}$. 
\end{enumerate}
\end{lemma}

\begin{proof}
For (\ref{it:aff.Xb}), suppose that $X$ is defined by the equations $f_1(x)= \ldots= f_N(x) =0, g(x)\neq 0$ with $f_1, \ldots, f_N, g\in \CC[t][x_1, \ldots, x_n]$. Taking $x_i = \sum_{j = 0}^{b-1} a_{ij}t^j$ with $a_{ij}\in \CC$, we obtain the equations for $X(b)$ by expanding every $f_k(x)$ into powers of $t$ and requiring that all coefficients coming from $f_1, \ldots, f_N$ vanish as polynomials in $(a_{ij})_{ij}$,  and at least one coefficient coming from $g$ does not vanish as a polynomial in $(a_{ij})_{ij}$.

For (\ref{it:proj.Xb}), we consider $\AA^{b(n+1)}_{\CC}$ as the space of $(n+1)$-tuples of polynomials of degree strictly smaller than $b$. Let $U\subset \AA^{b(n+1)}_{\CC}$ be the subset of such $(n+1)$-tuples which are coprime. Using non-vanishing of resultants, $U$ is an open and dense subset of $\AA^{b(n+1)}$. It is also clear that $U$ is a cone with cone point $0$. We can then define $X(b)$ as a subset of $U / \CC^\times \subset \PP^{b(n+1)-1}_\CC$, and reason similarly as in the affine case.
\end{proof}

Given $X$ as above, we are interested in the geometric properties of $X(b)$.  Evidently, if $X$ is  defined over $\CC$, then we see that $X(1)$ is  isomorphic to $X(\CC)$. Note that even if $X$ is smooth, $X(b)$ may be badly behaved (see Remark \ref{rem: comparison to jets}).

\begin{definition}
Let $X \subset \PP^n_{K}$ be a quasi-projective variety. We define 
\[
N(X,b):= \dim X(b).
\]
If $X\subset \AA^n_K$ is a quasi-affine variety, then we set
\[
\Naff(X,b):= \dim X(b). 
\]
\end{definition}

The functions $N(X,b)$ and $\Naff(X,b)$ are the analogues of the usual counting functions over global fields.  Namely, they measure the density of $\CC(t)$-points or $\CC[t]$-points lying on $X$. 
\begin{remark}
Note that given a variety $X \subseteq \PP^n_{\CC(t)}$ and $b \ge 1$,
the variety $X(b)$ depends on the embedding of $X$ into $\PP^n_{\CC(t)}$. For example, $C=\{ yz^{d-1}=x^d\} \simeq \PP^1_{\CC(t)}$ but $C(b)$ is very different from $\PP^1_{\CC(t)}(b)$ (see Example \ref{ex: proj curve with many points}). 
\end{remark}
%

\subsection{A trivial bound}

We give a geometric version of the well known Schwartz--Zippel bounds, which include a bound on the number of irreducible components of largest possible dimension. 
The dimension bound is also proven in~\cite[Lem.\,5.1.1]{CCL-PW}.

\begin{proposition}[Geometric Schwartz--Zippel]
\label{prop: affine geometric SZ}
Let $X \subseteq \AA^n_{K}$ be an affine variety of pure dimension $m$ and degree $d$. 
Then for every $b \ge 1$
 we have 
\[
\Naff(X,b) \le mb.
\]
Moreover, the number of irreducible components of $X(b)$ of dimension $mb$ is at most $d$.
\end{proposition}

\begin{proof}
We may assume  $X$ has strictly positive dimension, as the result is trivial otherwise. 
Let $Z$ be an irreducible component of $X$. Choose a coordinate $x_j$ such that $Z$ is not contained in $\{x_j=\lambda\}$ for any $\lambda \in K$, which is possible as $Z$ is irreducible of positive dimension. Consider the map $\pi_j: Z\to \AA^1$ which projects onto the $j$-th coordinate. 
For every $y\in K$, the fibre $\pi_j^{-1}(y)$ is a variety of dimension  $m-1$ and of degree at most $\deg Z$, if it is not empty. 
For each $b$, we get map $\pi_{j,b}:Z(b) \to \AA^1_K(b)\simeq \AA^b_\CC$, where $\pi_{j,b}^{-1}(y)\simeq Z_{\tilde{y}}(b)$ for a suitable $\tilde{y}\in K$ of height smaller than $b$.  
By induction, 
for each $y \in \CC^b$ 
we get $\dim \pi_{j,b}^{-1}(y) \le (m-1)b$ with at most $\deg(Z)$ irreducible components of dimension $(m-1)b$. The claim now follows as 
\[
\Naff(Z,b) \le b+(m-1)b=mb,
\] 
and as the number of irreducible components of $Z(b)$ of dimension $mb$ is bounded by $\deg Z$ (note that
 $\pi_{j,b}$ is dominant when restricted to each irreducible component of dimension $mb$ of $Z(b)$ by dimension considerations).
\end{proof}

Note that Proposition~\ref{prop: affine geometric SZ}  assumes no irreducibility of $X$, and hence the estimate is tight when $X$ is a union of linear spaces defined over $\CC$. 
Moreover, this proposition also holds when $X$ is not defined over $K$.

Given a projective variety $X \subset \PP_{K}^n$, we get the following  projective version by applying Proposition~\ref{prop: affine geometric SZ} to the affine cone of $X$. 

\begin{corollary}\label{cor: projective geometric SZ}
Let $X \subset \PP^n_{K}$ be a projective variety of pure dimension $m$ and degree $d$. 
Then for every $b \ge 1$
we have 
\[
N(X,b) \le (m+1)b-1,
\]
and the number of irreducible components of $X(b)$ of dimension $(m+1)b-1$ is at most $d$.
\end{corollary}

\begin{proof}
Let $Y\subset \AA^{n+1}_K$ denote the affine cone of $X$. This is a variety of dimension $m+1$ and degree $d$, and hence we conclude by noting that
\[
N(X,b)= \Naff(Y,b)-1
\]
and using Proposition~\ref{prop: affine geometric SZ}. 
\end{proof}

We deduce the following about existence of $\CC$-points not lying on a given variety defined over $K$. 
This result  is frequently  used to show that 
given a generic condition over $K$, there already exists a $\CC$-point satisfying the condition.
See for example Proposition~\ref{prop: Bertini} or the proof of Lemma~\ref{lem:lines.aff.surface}.
This claim  is also given in~\cite[Corollary 5.1.2]{CCL-PW}.

\begin{corollary}
\label{cor: existence of C points}
Let $X \subset \AA^n_K$ be a quasi-affine variety with $\dim X < n$. Then there exists a point $y \in \CC^n$ not belonging to $X(1)$. An analogous statement holds in the projective case. 
\end{corollary}

\begin{proof}
By Proposition~\ref{prop: affine geometric SZ}, $\dim X(1) \le \dim X$,
whereas $\dim (\AA^n_K(1))=n>\dim X$, and so there exists a point $y\in \AA^n_K(1)\setminus X(1)$.
\end{proof}

It is typically enough to deal with geometrically  irreducible varieties. 

\begin{lemma}\label{lem:irre.vs.abs.irre}
Let $X$ be an affine or projective variety of dimension $m$ over $K$ which is integral over $K$, but not geometrically integral. Then there exists a subvariety $Y\subset X$ of dimension $m-1$, and of degree at most $d^2$, such that for all integers $b \ge 1$ we have 
\[
X(b)\subset Y(b).
\]
In particular, if $X$ is projective then
\[
N(X,b)\leq mb-1,
\]
and the number of irreducible components of $X(b)$ of dimension $mb-1$ is at most $d^2$.
If $X$ is affine then
\[
\Naff(X,b)\leq (m-1)b,
\]
and the number of irreducible components of $X(b)$ of dimension $(m-1)b$ is at most $d^2$.
\end{lemma}

\begin{proof}
This follows similarly to the argument below~\cite[Cor.\,1]{Heath-Brown-Ann}. 
The final part follows by applying Corollary~\ref{cor: projective geometric SZ} or Proposition~\ref{prop: affine geometric SZ} to $Y$.
\end{proof}

\subsection{Projections}

To reduce to the case of hypersurfaces in our main results we use projections.

\begin{proposition}
\label{prop: projection for affine varieties} 
Let $X \subset \AA^n_K$ be a geometrically integral affine variety of degree $d$ and dimension $m$. 
Then there exists a degree $d$ hypersurface $Z\subset \AA_K^{m+1}$ birational to $X$, and such that for every $b \ge 1$ the varieties $X(b)$ and $Z(b)$ are birational.
In particular
\[
N_{\aff}(X,b) = N_{\aff}(Z,b).
\]
\end{proposition}

\begin{lemma}
\label{lem: projection for proj varieties}
Let $X \subset \PP^n_K$ be a geometrically integral  variety of dimension $m$ and degree $d\ge 2$ and let $L \subset \PP^n_K$ be a hyperplane defined over $\CC$ not containing $X$. 
Then there exists an $(m+1)$-dimensional coordinate space $\Gamma =\{ x_{i_1}=\ldots =x_{i_{n-m-1}}=0  \} \subset \PP^n_K$ and a birational projection map $\pi: X \to Y $ to a degree $d$ geometrically integral hypersurface $Y \subset \Gamma$, such that for every $ x\in X(K)$ we have,
\[
h(\pi(x)) \le h(x).
\]
In addition, the map $\pi$ can be chosen such that  for every $x \in X$, 
 \[
	x \in L \iff \pi(x) \in L.
 \]  
\end{lemma}
\begin{proof}
We prove the claim by induction on $n$. 
If $X$ is a hypersurface there is nothing to prove, so assume $m \le n-2$. 
For a choice of a generic point $p \notin X$, the projection of $X$ from $p$ to any hyperplane $\Gamma \subset \PP^n_K$ not containing $p$ is birational to its image,  which is a degree $d$ projective variety (see e.g.~\cite[Page 8]{BHB06} or \cite[Example 18.16]{Harris}).

More precisely, by \cite[Page 224]{Harris}, if we fix any $q$ on $X \backslash L$ then it is enough to choose $p$ not lying on the cone over $X$ with vertex $q$, which we denote by $C_{X,q}$. 
Note that $C_{X,q}$  does not contain $L$, 
as it is irreducible of dimension at most $n-1$, and contains $q \notin L$. 
Choose any $q \in X$, set $L_i=\{ x_i=0\}$ and let $U$ be the complement of $C_{X,q} \cup L_0 \cup \ldots \cup L_n$ in $L$ (if $L$ is a coordinate hyperplane, exclude $L$ from the union $L_0 \cup \ldots \cup L_n$ above). 
Then $U$ is an open set in $L$, and by Corollary \ref{cor: existence of C points} and the above, we may find a  point $p \in U$ defined over $\CC$ and 
 a coordinate hyperplane $L_i$ for some $1 \le i \le n$, 
 to obtain the desired projection $\pi: X \to Y \subset L_i$. 
Without loss of generality, assume $i=n$ and set $\Gamma=L_n$. 
Given $x \in X$, the image $\pi(x)$ is the intersection of $\Gamma$ with the line connecting $p$ and $x$ and is given by the formula
\begin{equation}
\label{eq: projection equation}
\pi(x)
=\left(x_0-\frac{p_0}{p_n}x_n,\ldots, x_{n-1}-\frac{p_{n-1}}{p_n}x_n,0\right).
\end{equation}
Since $p$ is defined over $\CC$, 
 the inequality $h(\pi(x)) \le h(x)$ therefore holds.
Finally, note that since $p \in L$, 
and $L$ contains any line between two points in $L$, 
we get 
\[\pi(x) \in L \iff x \in L.
\]
  We conclude by induction, as $Y \subset \Gamma$ and $\dim \Gamma =n-1$.
\end{proof}
\begin{proof}[Proof of Proposition \ref{prop: projection for affine varieties}] 
Consider the projective closure $\overline{X}$ of $X$ in $\PP^n_K =\AA^n_K \cup L_\infty$,
 where $L_\infty =\{ x_0=0\}$ is the hyperplane at infinity. By Lemma \ref{lem: projection for proj varieties}, there exists a coordinate linear space $\Gamma \subset \PP^n_K$ of dimension $m+1$ and a regular birational map 
$\pi : \overline{X}\to Y $ to a degree $d$ hypersurface $Y\subset \Gamma$ which does not increase the height of each point $x \in \overline{X}(K)$. 
In addition, by the lemma we may choose $\pi$ such that $\pi(X)$ 
does not intersect the plane at infinity  $L_\infty=\{ x_0=0\}$ in $\PP^n_K$. 
Without loss of generality, assume $\Gamma =\{ x_n=0\}$. 
Set $Z:=\overline{\pi(X)}$. Then $Z$ is the required affine hypersurface. 
By Equation~\eqref{eq: projection equation}, since we project from points defined over $\CC$, 
for each $b$ we get a birational projection map $\pi_b : X(b) \to Z(b)$. 
In particular, $N_{\aff}(X,b)= N_{\aff}(Z,b)$. 
\end{proof}

\subsection{Linear spaces}

In this section we give a precise count on points of bounded height on linear spaces, which will be used when counting on lines on a surface. We give a proof using geometry of numbers in function fields, although one can also prove this result via elementary means.

We denote by $\Gr(m,n)$ the Grassmannian of $m$-dimensional vector subspaces of $K^n$. For $L\in \Gr(m,n)$, we denote by $h(L)$ the height of the point $L$ under the Pl\"ucker embedding $\Gr(m,n)\to \PP^N$.
This quantity is also often called the determinant of $L$.

\begin{lemma}\label{lem:linear.spaces}
Let $L\in \Gr(m,n)$. Then for $b\geq h(L)$ we have that
\[
\Naff(L,b) = bm - h(L).
\]
Moreover, for $b> h(L)/m$ we have that $L(b)\neq \varnothing$.
\end{lemma}

\begin{proof}
Let $v_1, \ldots, v_m\in \cO_K^n$ be a reduced basis of $L$, i.e.~ordered with increasing height such that each  $v_i$ realizes  the successive minimum $s_i = h(v_i)$ of $L$ (see \cite[Lemma 4.1]{Hess} or \cite{Paulus}). 
Then it follows from~\cite[Lem.\,4.1]{Hess} that 
\begin{enumerate}
\item for $\lambda_i\in \cO_K$ we have $h(\sum_{i=1}^m \lambda_i v_i) = \max\limits_{i: \lambda_i \neq 0} \left(s_i + h(\lambda_i)\right)$, and
\item $h(L)= s_1 + \ldots + s_m$.
\end{enumerate}
The first property shows that
\[
\Naff(L,b) = \sum_{i=1}^m \max\{b-s_i, 0\}.
\]
Hence for $b \geq h(L)\geq s_m$, we obtain that $\Naff(L,b) = bm - h(L)$. Moreover, for $b > h(L)/m \geq s_1$ we have that $\Naff(L,b) > 0$, as desired.
\end{proof}

\begin{lemma}\label{lem:kernel.linear.spaces}
Let $m<n$ and let $A\in K^{m\times n}$ be a matrix of full rank. 
Let $L\in \Gr(n-m, n)$ be the kernel of $A$. 
Consider $A$ as a point in $\Gr(m,n)$ 
by considering the rows of $A$ as vectors in $K^n$. 
We then have that
\[
h(L) = h(A).
\]
\end{lemma}

\begin{proof}
After permuting the coordinates and Gaussian elimination, we may assume that $A$ is of the form
\[
A = \begin{pmatrix} I_m & \mid B\end{pmatrix}.
\]
The claim now follows as the columns of the following matrix form a basis of $L$
\[
\begin{pmatrix}
-B^T \\
I_{m}
\end{pmatrix}.\qedhere
\]
\end{proof}

We will also need an analogue of the Bombieri--Vaaler theorem~\cite{Bombi-Vaal} for function fields, which follows from the above results. 

\begin{lemma}[Bombieri--Vaaler]\label{lem:BV}
Let $A\in \cO_K^{m\times n}$ be a matrix with $m<n$. Assume that $A$ has full rank. Then there exists a non-zero $x\in \cO_K^n$ with $Ax = 0$ and 
\[
h(x) \leq \frac{h( \det(A\overline{A}^T) )/2 - h(D)}{n-m},
\]
where $D$ is the greatest common divisor of all $m\times m$ minors of $A$. 
\end{lemma}

\begin{proof}
Let $L\in \Gr(n-m, n)$ be the linear space corresponding to the kernel of $A$. Then Lemma~\ref{lem:linear.spaces} shows that there exists a nonzero point $x\in \cO_K^r$ with $h(x)\leq h(L)/(n-m)$. The result then follows from~Lemma~\ref{lem:kernel.linear.spaces} after noting that
\[
h(L) = h(A) = h( \det(A\overline{A}^T) )/2 - h(D). \qedhere
\]
\end{proof}

\subsection{Points on curves over function fields}

To prove Theorem \ref{thm:main.siegel}, we use induction on the dimension. The base for the induction is 
Proposition~\ref{prop:curve.siegel} about sparseness of $\cO_K$-points lying on curves. 
For this aim, we need an analogue of Siegel's theorem over function fields, which we recall from the literature below.
The following is a result due to Mason.

\begin{theorem}[{\cite[Theorems 8, 11]{Mason_1984}}]
\label{thm:Mason}
Let $C \subset \AA^2_K$ be a geometrically irreducible curve.  
\begin{enumerate}
\item Assume that $C$ has genus $0$ and at least three distinct points at infinity. Then $C$ has infinitely many $\cO_K$-points  if and only if there exists 
a rational parametrization 
$\psi: \AA^1_K \to C$ of the form
\[
\psi(x)=\left(\frac{a(x)}{c(x)}, \frac{b(x)}{c(x)}\right)
\]
where $a,b \in \cO_K[x], c \in \CC[x]$ have no common factors. Furthermore, if such a $\psi$ exists, then there are only finitely many $\cO_K$-points on $C$ which are not the image under $\psi$ of a $K$-point on $\AA^1_K$.
\item Assume that $C$ has genus $1$. 
Then $C$ has infinitely many $\cO_K$-points if and only if there exists an elliptic curve $D\subset \AA^2_K$ defined by $y^2 = j(x)$ for a square-free monic polynomial $j\in \CC[x]$ and a parametrization $\psi: D\to C$ of the form
\[
\psi(x,y) = \left(\frac{a(x) + yb(x)}{c(x)}, \frac{d(x) + ye(x)}{g(x)}\right),
\]
where $a,b,d,e\in \cO_K[x]$ and $c,g\in \CC[x]$.
Furthermore, in this case the $\cO_K$-points of $C$ consist of the infinite family arising  by  taking $x \in \CC$, and finitely many other $\cO_K$-points.
\end{enumerate}
\end{theorem}

To take care of the case of genus at least $2$, we recall the resolution of Mordell's conjecture (Falting's theorem) over function fields. 
\begin{theorem}[Manin--Grauert, see \cite{Samuel66}]
\label{thm:Manin-Grauert}
Let $C \subset \AA^2_K$ be a geometrically irreducible curve of genus at least $2$ which is not defined over $\CC$. Then $C(\cO_K)$ is finite. 
\end{theorem}
\subsection{Examples}
In this section we provide examples which demonstrate some noteworthy phenomena arising in our geometric version of dimension growth. 

Firstly, 
we show the upper bounds on dimension given in Theorems \ref{thm: affine curves} and \ref{thm: projective curves} are tight, where the curves attaining these bounds also attain analogous bounds in classical dimension growth. 
\begin{example}[{See \cite[Remark 5.1.5]{CCL-PW}}]
\label{ex: affine curve with many points}
Let $d \ge 1$ and let $C \subset \AA^2_K$ denote the affine  curve defined by $y=x^d$. Then for every integer $b \ge 1$ we have
\[
\Naff(C,b)= \left\lceil \frac{b}{d}\right\rceil,
\]
and the variety $C(b)$ is irreducible. Indeed, it is the image in $\AA^2_K(b)$ of the irreducible variety $\AA^1_K\left(\left\lceil \frac{b}{d}\right\rceil\right)$ under the map arising from $x \mapsto (x,x^d)$.
Note however that $C(b)$ is not reduced for every $b\geq 2$.
\end{example}

\begin{example}
\label{ex: proj curve with many points}
Let $d \ge 1$ and let $C \subset \PP^2_K$ denote the projective  curve defined by $yz^{d-1}=x^d$. Then for every integer $b \ge 1$ we have
\[
N(C,b)= 2\left\lceil \frac{b}{d}\right\rceil-1,
\]
and the variety $C(b)$ is irreducible.
\end{example}

In Theorems~\ref{thm: affine curves} and~\ref{thm: projective curves} we only count the number of irreducible components of maximal possible dimension. 
Accessing lower-dimensional irreducible components seems quite difficult to us, and one should not expect any good bound on their number, as the following example illustrates.

\begin{example}
For each $n \ge 1$, 
 let $X_n$ be the vanishing locus of the generalized Pell equation $y^2-(t^2+t+1)x^2=(t-1)(t-2)...(t-n)$ in $\AA^2_K$.
Let $L$ be the function field $K(\sqrt{t^2+t+1})$ and let $\phi: L\to L$ be the non-trivial automorphism of $L$ over $K$.

By Lemma \ref{le:pell}, $\dim X_n(b)=0$ for every $n, b$.
Hence each $X_n(b)$ consists of a finite number of points.
However, we claim that 
\[
\# X_n(b)\ge 2^{n}, \quad\text{ for }b=n+1.
\]
Indeed, it is easy to see that for each $i\in \CC$, the equation $y^2-(t^2+t+1) x^2=t-i$ has a solution $(x_i, y_i) \in \cO_K^2$ with $\deg x_i=0$, $\deg y_i=1$. 
Now, for each function $\sigma: \{1, ..., n\} \to \{1,-1\}$ we get an element
\[
y_\sigma+x_\sigma\sqrt{t^2+t+1} =  \prod_{i=1}^n (y_i+\sigma(i) x_i\sqrt{t^2+t+1} ) \in L
\]
for which $(x_\sigma, y_\sigma) \in X_n(n+1)$.

We claim that these $(x_\sigma, y_\sigma)$ are all different.
To see this, assume that $(x_\sigma, y_\sigma) = (x_\tau, y_\tau)$ for two functions $\sigma,\tau$.
Let $S\subseteq \{1, \ldots, n\}$ be the set of $j$ for which $\sigma(j) \neq \tau(j)$ and define
\[
\alpha_S = \prod_{i\in S}(y_i + \sigma(i) x_i\sqrt{t^2+t+1} )\in L.
\]
By assumption $\phi(\alpha_S) = \alpha_S$.
Therefore the norm of $\alpha_S$ is equal to $\alpha_S^2 = \prod_{i\in S}(t-i)$.
Each point $t-i$ for $i=1, \ldots, n$ is unramified in $L$, and so unique factorization into prime ideals implies that $S = \varnothing$.
We conclude that $\#X_n(n+1) \ge 2^{n}$, as desired.
\end{example}




Even if $X$ enjoys nice regularity properties, e.g.\,$X$ is a smooth variety, the scheme $X(b)$ may be very singular. In this respect, $X(b)$ is very different from the jet schemes of $X$ (see Remark~\ref{rem: comparison to jets}).

\begin{example}
\label{ex: non-reduced C(b)}
Let $C$ denote the curve in $\PP^2_{K}$ cut by $tx^2=yz$. Then $C(b)$ is non-reduced for every integer $b \ge 1$. 
\end{example}
\begin{proof}
Given $b \ge 1$, the scheme $C(b)$ is an open sub-scheme of the closed scheme $Y$ cut out in $\PP^{3b-1}_{\CC}$ by the equations obtained from the relation $tx^2=yz$. 
Setting $x= x_0 + \ldots +x_{b-1}t^{b-1}$, we see that the homogeneous ideal 
 $\cI_Y$ defining $Y$ contains the element $x_{b-1}^2$, but since it is generated by homogeneous degree $2$ polynomials, $ x_{b-1} \notin \cI_Y$. 
It follows that $C(b)$ is non-reduced. 
\end{proof}

Finally, given a variety $X$ and an integer $b\ge 1$, we are interested in the number of irreducible components of top dimension of $X(b)$. The next example shows that this number may strictly increase in $b$ even if $X$ is a smooth variety. 
\begin{example}
\label{ex: affine variety with many irred components}
Let $X$ denote the affine surface in $\AA^3_K$ cut out by the equation $ xy=z$ and let $b \ge 1$ be an integer. Then $X(b)$ has dimension $b+1$ and $b$ irreducible components of top dimension (note that $X$ does not satisfy the assumptions of Proposition \ref{prop:aff.surfaces}).
\end{example}
\begin{proof}
Consider the projection map $p_b:X(b) \to \AA^{b}_{\CC}$ defined by $p_b(x,y,z)= z$. If we fix $z\in \cO_K$ non-zero, then up to multiplication by $\CC^\times$ there are only finitely many $x,y\in \cO_K$ such that $xy=z$. Hence every fibre of $p_b$ not over $z=0$ is one-dimensional (and the fibre over $0$ is of dimension $b$).
We therefore conclude that $\dim X(b)=b+1$. 
Now note that for each $ 0 \le i \le b-1$, the subscheme 
\[
X_i:=\{ (x,y,z) \in X(b) : x_{j}=y_k=0, \, j > i, k>b-1-i\}
\]
 is irreducible of dimension $b+1$ since the projection $p_i : X_i \to \AA^{i+1} \times \AA^{b-i}$ by $p_i(x,y,z)=(x,y)$ is an isomorphism.
\end{proof}

\begin{remark}
\label{rem: comparison to jets}
Let $X\subset \AA^n_K$ be an affine variety defined over $\CC$.
Recall that the $\CC$-points of the $k$-th jet scheme $J_k(X)$ of a scheme $X$ are given by $X(\CC[t]/(t^{k+1}))$, see e.g.\,\cite{EM09} for more details.
The space $X(b)$ is therefore naturally a subvariety of the $(b-1)$-st jet scheme $J_{b-1}(X)$ of $X$. While $ J_{b-1}(X)$ reflects the singularities of the variety $X$, the variety $X(b)$ might be badly behaved and may have many irreducible components even if $X$ is smooth, as seen in Examples \ref{ex: non-reduced C(b)}, \ref{ex: affine variety with many irred components}. 
In contrast, if $X \subset \AA^n_K$ is an irreducible  local complete intersection  with at most rational singularities, then $J_{b-1}(X)$  is irreducible of dimension $b\dim X$ by \cite{Mustata1}. 
\end{remark}
Finally, we show that the dimension bounds from Theorems \ref{thm:main.siegel}	and 	\ref{thm:main.dgc} are attained. 
\begin{example}
\label{ex: attaining upper bound}
	Consider an irreducible curve $C \subset \Gr(m+1, n+1)$
	of degree at least $2$ defined over $\CC$ and let 
	\[
	X=\{(p, W) \in \PP^n \times \Gr(m+1,n+1) : p \in \PP(W)\}
	\]	
	be the incidence variety corresponding to the Grassmannian.  	Let 
$p_1: X \to \PP^n$ and 	
	$p_2: X \to \Gr(m+1,n+1)$ denote the natural projections, and consider $Y_C=p_1(p_2^{-1}(C))$. 
	Assume that a general point in $Y_C$ lies in a unique linear space in $C$. Then  $\dim Y_C=m+1$ and $\deg Y_C = \deg(C)$ (see e.g.~\cite[Example 19.11]{Harris}). 
	Moreover, given $b \ge 1$ it is straightforward to verify that $\dim Y_C(b) \ge (m+1)b$, therefore $Y_C$ attains the dimension bound given in Theorems \ref{thm:main.siegel}	and 	\ref{thm:main.dgc}. 
\end{example}

%
\section{Integral points on curves}
\label{sec: counting points on curves}
In this section we prove Proposition~\ref{prop:curve.siegel}, which will form the base of the induction for Theorem~\ref{thm:main.siegel}. 

We first need a result about Pell equations, for which we briefly deal with more general function fields.
Let $L$ be a function field of a curve $C$ over $\CC$, and let $S\subset C(\CC)$ be a finite set of points. We will write the elements of $S$ as valuations on $L$. Denote by $\cO_{L,S}$ the ring of $S$-integers of $L$, which consists of all elements $a\in L$ for which $v(a)\geq 0$ for every $v\not\in S$.
For an element $a\in L$, recall that the height of $a$ is defined as
\[
h(a) = \deg(a),
\]
where $a$ is considered as a morphism $a: C\to \PP^1$. 
When $L=K$ and $S = \{\infty\}$, this height agrees with the height as defined in the previous section.

\begin{lemma}
Let $b$ be a positive integer. Then the set
\[
\cO_{L,S}^\times (b) = \{a\in \cO_{L,S}^\times: h(a)<b\}
\]
is finite up to multiplication by $\CC^\times$.
\end{lemma}

\begin{proof}
Let $S = \{s_1, \ldots, s_n\}$ and let $a$ be an element of $\cO_{L,S}^\times (b)$. Then all zeroes and poles of $a$ are contained in $S$. Denoting by $e_i$ the order of $a$ at $s_i$, we obtain that
\[
\deg a = \sum_{e_i\geq 0} e_i = \sum_{e_i<0}-e_i < b.
\]
Hence there are only finitely many options for $e_1, \ldots, e_n$. We conclude by noting that an element of $L$ is determined by its zeroes and poles (with multiplicities), up to multiplication by $\CC^\times$.
\end{proof}

We need some auxiliary results about equations of Pell type to deduce Proposition~\ref{prop:curve.siegel}.

\begin{lemma}[Pell equation]\label{le:pell}
Let $\beta, \gamma$ be in $\cO_{K}$ and let $b$ be a positive integer. Assume that $\beta$ is not a square in $K$ and that $\gamma$ is non-zero. Then the equation 
\[
x^2 - \beta y^2 = \gamma
\]
has only finitely many solutions $a\in \cO_{K}^2$ with $h(a)\le b$.
\end{lemma}

\begin{proof}
Let $L = K(\sqrt{\beta})$, which is the function field of an irreducible curve $D$ since $\beta$ is not a square in $K$. The curve $D$ comes with a degree two cover $D\to C$. Let $S\subset D$ consist of all points above $\infty$ and all points above zeroes of $\gamma$. Let $(x_1, y_1)$ and $(x_2, y_2)$ be two solutions of $x^2-\beta y^2 = \gamma$ in $\cO_{K}^2$ of height at most $b$. 
Then the elements $x_i + y_i\sqrt{\beta}$ in $\cO_{L,S}$ have height at most $2b + h(\sqrt{\beta})$, and so the element $z = (x_1+y_1\sqrt{\beta})/(x_2+y_2\sqrt{\beta})$ has height bounded by $b'=4b+2h(\sqrt{\beta})$. 
But $z$ is also a unit in $\cO_{L,S}$ and so $z\in \cO_{L,S}^\times(b')$. We conclude by the previous lemma, and the fact that the equation $x^2 - \beta y^2 = \gamma$ is non-homogeneous (otherwise the set of solutions is a cone).
\end{proof}



To prove Proposition~\ref{prop:curve.siegel}, we will distinguish several cases.
Assume first that $C$ has genus at least one. 
If $C$ is not defined over $\CC$, then Siegel's theorem over function fields gives finiteness of $C(\cO_{K})$ (Theorems \ref{thm:Mason}(2) and \ref{thm:Manin-Grauert}). If $C$ is defined over $\CC$, then the following geometric argument shows that $C(K) = C(\CC)$.

\begin{lemma}\label{lem:curvesOverC}
	Let $C \subseteq \AA^2_\CC$ be a geometrically integral curve defined over $\CC$ of genus at least $1$. Then $C(K)=C(\CC)$. 
	In particular, $\dim C(b)\leq 1$ for every positive integer $b$.
\end{lemma}

\begin{proof}
Let $C_s\to C$ be the normalization of $C$, which is a smooth projective curve birational to $C$. Any $K$-point of $C$ corresponds to a morphism $\PP_\CC^1\to C$, which yields a map $\PP_\CC^1\to C_s$ since $\PP^1$ is normal. As $C_s$ has genus at least $1$, this morphism must be constant, i.e.~it must correspond to a $\CC$-point on $C$. 

For the final statement simply note that every $\cO_{K}$-point is also a $K$-point.
\end{proof}
%
%

For the genus $0$ case, we will need the following classical lemma to parametrize $C$.
We embed $\AA^2_K$ in $\PP^2_K$ so that the line at infinity is given by $x_2 = 0$.

\begin{lemma}\label{lem:gen0Parametr}
	Let $C \subset \mathbb{P}_K^2$ be a geometrically integral projective curve over $K$.
	Suppose that $C$ has genus $0$ and that $C(K)$ has at least three smooth points.
	Then there exist homogeneous polynomials $P_1, P_2, P_3 \in \cO_K[s,t]$ all of the same degree such that
	\[
	\mathbb{P}^1_K \to C: [s:t] \mapsto [P_1(s,t):P_2(s,t):P_3(s,t)]
	\]
	is a birational map. 
	Moreover, if $C$ has at least three points at infinity and infinitely many points of the form $[x:y:1]$, with $x, y \in \cO_K$, then we can take $P_3$ to be a polynomial over $\CC$.
\end{lemma}

\begin{proof}
	Let $\overline{K}$ be the algebraic closure of $K$ and let $C_s\to C$ be the normalization of $C$.
	Then $C_{s, \overline{K}}$ is an irreducible smooth, projective, genus $0$ curve and hence isomorphic to $\mathbb{P}_{\overline{K}}^1$. 
	Therefore we can find an isomorphism $\psi_1:\mathbb{P}_{\overline{K}}^1 \to C_{s, \overline{K}}$, which leads to a birational map
	\[
	\psi:\mathbb{P}_{\overline{K}}^1 \to C_{\overline{K}}: [s:t] \mapsto [P_1(s,t):P_2(s,t):P_3(s,t)],
	\]
	where the $P_i$ are homogeneous polynomials over $\overline{K}$.
	Moreover, by composing with an automorphism of $\mathbb{P}_{\overline{K}}^1$ we may assume that $\psi([0:1])$, $\psi([1:0])$ and $\psi([1:1])$ are smooth $K$-points on $C(K)$.
	We may also assume that $P_3$ has $1$ as one of its coefficients.
	
	We now claim that $P_i$ have coefficients $K$. 
	To see this, consider any $\sigma \in \text{Gal}(\overline{K}/K)$.
	Let $\psi^{\sigma}: \mathbb{P}_{\overline{K}}^1 \to \mathbb{P}_{\overline{K}}^2$ be the map $\psi$ with all the coefficients of the $P_i$ replaced by their images under $\sigma$.
	Note that $\psi^{\sigma}\left( \mathbb{P}_K^1\right)=\sigma\left(\psi\left(\mathbb{P}_K^1\right)\right)$.
	As $C$ is defined over $K$, the latter is actually a subset of $C_{\overline{K}}$.
	As a consequence, the image of $\psi^{\sigma}$ is also $C_{\overline{K}}$.
	
	Now note that $\psi^{\sigma}$ pulls back to a map $\psi^{\sigma}_1: \mathbb{P}_{\overline{K}}^1 \to C_{s, \overline{K}}$.
	Moreover, because $\psi^{\sigma}$ is injective on an open subset of $\mathbb{P}_K^1$, $\psi^{\sigma}_1$ has degree $1$ and is therefore an isomorphism.
	Hence $\psi_1^{\sigma}=\psi_1 \circ \alpha$ for some automorphism $\alpha$ of $\mathbb{P}_K^1$.
	Because both $\psi$ and $\psi^{\sigma}$ take the same values on $[0:1]$, $[1:0]$ and $[1:1]$, it follows that $\alpha$ is the identity, so $\psi=\psi^{\sigma}$.
	As $P_3$ had $1$ as one of its coefficients, it follows that $P_i^{\sigma}=P_i$. 
	We conclude $P_i$ are defined over $K$.
	The final claim follows from Theorem~\ref{thm:Mason}.	
\end{proof}

\begin{proof}[Proof of Proposition~\ref{prop:curve.siegel}]
Assume first that $C$ has genus at least $1$.
If $C$ is defined over $\CC$, then the result follows from Lemma~\ref{lem:curvesOverC}.
So assume that $C$ is not defined over $\CC$. 
If $C$ has genus at least $2$ then the result follows from~Theorem~\ref{thm:Manin-Grauert}.
If $C$ has genus $1$, then by Mason's theorem \ref{thm:Mason}, the $\cO_K$-points consist of an infinite family parameterized by $\CC$, and finitely many other ones, so we get $\dim C(b) \le 1$ as well. 

We may therefore assume that $C$ has genus $0$.
By Lemma~\ref{lem:gen0Parametr}, there exist homogeneous polynomials $P_1, P_2, P_3\in \cO_K[s,t]$ such that all but finitely many integral points on $C$ are of the form
\[
\left(\frac{P_1(u,v)}{P_3(u,v)}, \frac{P_2(u,v)}{P_3(u,v)}\right).
\]
Now we need to find $[u:v] \in \PP_K^2$ for which $P_1(u,v)/P_3(u,v)$ and $P_2(u,v)/P_3(u,v)$ are both elements of $\cO_K$.
We may rewrite these rational functions as $Q_1/R_1$, $Q_2/R_2$, where $Q_i\in \cO_K[u,v]$ and $R_i\in \cO_K[u,v]$ have no common factors over $K[u,v]$.
By taking the homogeneous resultant, we may find $A_i, B_i, C_i, D_i\in \cO_K[u,v]$ with
\[
A_i Q_i+B_i R_i=x_i u^{k_i}
\]
and
\[
C_i Q_i+D_i R_i=y_i v^{\ell_i},
\]
where $x_i, y_i \in \cO_K$ and $k_i, \ell_i \in \NN$.
Therefore, we see that if $Q_i(u,v)/R_i(u,v) \in \cO_K$ for a certain choice of $(u,v) \in \cO_K^2$, then $R_i(u,v) \mid x_i u^{k_i}$ and $R_i(u, v) \mid y_i v^{\ell_i}$.
Since we may assume that $u$ and $v$ have no common factors, we conclude that $R_i(u,v)\mid x_i, y_i$, proving that $R_i$ can only take finitely many values, up to a unit (i.e. an element of $\CC^\times$), as any element of $\cO_K$ only has finitely many divisors.

Assume now that $C$ has at least three distinct points at infinity.
Then we moreover have that $P_3$ and hence $R_1, R_2$ are polynomials over $\CC$.
Therefore, both of $R_1, R_2$ factor over $\CC$ into linear factors.
As there are at least two different points at infinity, at least two among these factors must be linearly independent, say $au+bv$ and $a'u+b'v$.
Now, because $R_1$ and $R_2$ can take only finitely many values up to a unit, the same is true for $a'u+b'v$ and $au+bv$.
Moreover, because the choice of $[u:v]$ is homogeneous in $u, v$, we may assume that $au+bv$ takes only finitely many values (by dividing out in $u, v$ by any potential unit).
Then $a'u+b'v$ takes values in a one-dimensional space, and therefore so do $u$, $v$, showing that indeed $\dim C(b) \le 1$.

Finally, if $C$ contains precisely two points at infinity, then the $R_i$ must be either of the form $L_2^{r_i}$ for $L_2\in \cO_K[s,t]$ of degree $2$ over $K$, or of the form $L_1^{r_1} L_2^{r_2}$ for two linearly independent linear forms.
In the former case, $L_2$ can take only finitely many values, and we conclude by Lemma~\ref{le:pell} (by completing the square); in the latter case, we can use precisely the same strategy as before.
\end{proof}

\section{The determinant method}
\label{sec: the determinant method}

Our main technique for proving Theorem~\ref{thm:main.dgc} is a geometric analogue of the determinant method. This method allows one to construct auxiliary hypersurfaces capturing  rational points of bounded height lying on $X$, and is the main topic of this section. Over global fields, this method has a long history with many improvements over the years. Let us mention the work of Bombieri--Pila~\cite{Bombieri-Pila} where the determinant method first appears, Heath--Brown's $p$-adic adaptation~\cite{Heath-Brown-Ann}, Salberger's global improvement using many primes~\cite{Salberger-dgc}, Walsh's work to keep track of the height of the polynomial~\cite{Walsh} and finally a version which is polynomial in $d$ due to Castryck--Cluckers--Dittmann--Nguyen~\cite{CCDN-dgc}. 

Our geometric determinant method is a mixture of these works.
It resembles Heath--Brown's $p$-adic determinant method for multiple primes at the same time, but with the improvements by Walsh and Castryck--Cluckers--Dittmann--Nguyen. 

We also prove Theorems \ref{thm: affine curves} and \ref{thm: projective curves} in this section. 

\subsection{A determinant lemma}

We need a lemma on divisibility of certain determinants.

For $X\subset \PP^{n+1}_{\cO_K}$ and $p\in \cO_K$ a non-zero prime element, we denote by $X_p\subset \PP^{n+1}_\CC$ the reduction of $X$ modulo $p$.
If $X$ is described by a set of polynomial equations over $\cO_K$ and $p = t-\lambda\in \cO_K$, then $X_p$ is obtained by simply substituting $t=\lambda$ in the equations for $X$.

\begin{lemma}[Divisibility of determinants]\label{lem:divisibility}
Let $X\subset \PP^{n+1}_{\cO_K}$ be a geometrically integral hypersurface, and for a non-zero prime element $p\in\cO_K$ let $X_p$ be the reduction of $X$ modulo $p$. Let $P$ be a $\CC$-point on $X_p$ of multiplicity $\mu$, and let $a_1, \ldots, a_s$ be $\cO_K$-points on $X$, given by primitive $\cO_K$-coordinates, that all reduce to $P$ modulo $p$. Let $F_1, \ldots, F_s$ be homogeneous polynomials in $x_0, \ldots, x_{n+1}$ over $\cO_K$. Then $\det(F_i(a_j))_{i,j}$ is divisible by $p^e$ with
\[
e\geq \left(\frac{n!}{\mu}\right)^{1/n}\frac{n}{n+1}s^{1+1/n} - O_n(s).
\]
\end{lemma}

\begin{proof}
This follows exactly as in~\cite[Lem.\,2.5]{SalbCrelle}. To make the upper bound independent of $d$, one follows~\cite[Cor.\,2.5]{CCDN-dgc}.
\end{proof}

\subsection{Auxiliary polynomials}

For the construction of our auxiliary polynomials, we need a coordinate transformation which keeps track of the height of our equations. 
For a homogeneous polynomial $f\in \cO_K[x_0, \ldots, x_{n+1}]$ denote by $c_f$ the coefficient of $x_{n+1}^d$ in $f$, and by $h(f)$ the height of the polynomial $f$ considered as a point in projective space $\PP^N(K)$. The following is clear. 

\begin{lemma}\label{lem:coordinate.transformation}
Let $f\in \cO_K[x_0, \ldots, x_{n+1}]$ be primitive and homogeneous of degree $d$ and set $X=V(f)$. 
Take $a_0, \ldots,a_n \in \CC$  such that 
\[
h(f(a_0, \ldots, a_n, 1)) = h(f),
\]
and consider the invertible linear transformation 
\[
A(x_0,\ldots,x_{n+1})=(x_0+a_0x_{n+1},\ldots,x_{n}+a_n x_{n+1},x_{n+1}).
\]
Then $g(x):=f(A(x))$ is a primitive and homogeneous
polynomial defined over $\cO_K$ 
 of degree $d$ such that $h(g)=h(c_g) = h(f)$.
 Moreover, denoting $Y=V(g)$,   the varieties $X(b)$ and $Y(b)$ are isomorphic for every integer $b \ge 1$.
%
\end{lemma}


\begin{theorem}[Projective auxiliary polynomials]\label{thm:proj.aux.poly}
Let $f\in \cO_K[x_0, \ldots, x_{n+1}]$ be absolutely irreducible, primitive and homogeneous of degree $d$, let $X$ be the variety cut out by $f$ in $\PP^{n+1}_{K}$, and let $b$ be a positive integer. Take positive integers $\mu_1, \ldots, \mu_\ell$ so that
\[
\sum_{i=1}^\ell \frac{1}{\mu_i^{1/n}}\geq \frac{(n+1)(b-1)}{nd^{1/n}}  - \frac{h(f)-1}{nd^{1+1/n}},
\]
and let $p_1, \ldots, p_\ell\in \cO_K$ be non-zero primes. 
Then if $P=(P_1,\ldots,P_\ell)\in \prod X_{p_i}(\CC)$ is a tuple where each $P_i$ is of multiplicity $\mu_i$, then there exists a polynomial 
\[
g:=g_{P}\in \cO_K[x_0, \ldots, x_{n+1}]
\]
coprime to $f$, of degree $O_n(d^{1+1/n}b + d^3\ell )$, and vanishing on
\[
X(b; P_1, \ldots, P_{\ell}) = \{x\in X(b): x\equiv P_i \bmod p_i\}.
\]
%
\end{theorem}
The most important use of this theorem is when all $\mu_i$ are equal to $1$. In that case, one simply takes $\ell$ primes $p_1, \ldots, p_\ell$ and smooth points $P_i\in X_{p_i}(\CC)$. An application of this case is enough to deduce Theorem~\ref{thm: projective curves}. 
However, to obtain Theorem~\ref{thm:main.dgc} for small values of $d$, namely $6\leq d\leq 9$, we will need to  apply this result with some singular points on the reductions $X_{p_i}$.

\begin{proof}
Denote by $c_f$ the coefficient of $x_{n+1}^d$ in $f$. Using the coordinate transformation from Lemma~\ref{lem:coordinate.transformation}, we can assume that $h(c_f)= h(f)$. 

Let $M$ denote the candidate degree of $g$, and denote by $B[M]$ the set of homogeneous monomials in $x_0, \ldots, x_{n+1}$ of degree $M$.
Note that $|B[M]| = \binom{M+n+1}{n+1}$.
Let $\Xi$ be a maximal subset of $X(b; P_1, \ldots, P_\ell)$ which is algebraically independent over $B[M]$ in the sense that applying the elements of $B[M]$ to $\Xi$ gives a matrix $A$ with linearly independent rows. So the columns of $A$ are parametrized by $B[M]$, while the rows are parametrized by $\Xi$. Let $r = |B[M]|$ and $s = |\Xi|$. If $g$ is a polynomial of degree $M$ vanishing on $\Xi$, then it actually vanishes on all of $X(b; P_1, \ldots, P_\ell)$. Assume that all degree $M$ polynomials vanishing on $X(b; P_1, \ldots, P_\ell)$ are divisible by $f$. We will obtain a contradiction when choosing $\ell$, the $\mu_i$, and $M$ as stated. For use below, we may already assume that $d^2 = O_n(M)$.

The matrix $A$ describes a linear system whose solutions are the degree $M$ homogeneous polynomials vanishing on $\Xi$. By assumption, the kernel of $A$ is exactly $f\cdot B[M-d]$, and so we obtain that $s=|B[M]| - |B[M-d]|$. Let $\Delta$ be the greatest common divisor of all $s\times s$ minors of $A$. By Lemma~\ref{lem:BV}, and the assumption on $c_f$, we have that
\[
h(\Delta) \leq h(\det(A\overline{A}^T))/2 -(r-s)h(f).
\]
On the other hand, the divisibility from Lemma~\ref{lem:divisibility} above shows that
\[
h(\Delta) \geq \frac{n!^{1/n}n}{n+1}s^{1+1/n} \sum_{i=1}^\ell\left(\frac{1}{\mu_i^{1/n}} - O_n(s^{-1/n})\right).
\]
Write $\kappa = \sum_i \frac{1}{\mu_i^{1/n}}$. Putting these together, we have that
\[
\frac{n!^{1/n}n \kappa }{n+1}s^{1+1/n} - O_n(\ell s)\leq h(\det(A\overline{A}^T))/2 - (r-s)h(f).
\]
The entries of $A$ are monomials of degree $M$ evaluated at points of height $<b$, and so $h(\det(A\overline{A}^T))/2 \leq Ms(b-1)$. Dividing by $Ms$ yields that
\[
\frac{n!^{1/n}n  s^{1/n}\kappa }{(n+1) M} - O_n(\ell/M) \leq b-1 - \frac{r-s}{Ms}h(f).
\]
Now 
\begin{align*}
s &= |B[M]| - |B[M-d]| = \binom{M+n+1}{n+1} - \binom{M-d+n+1}{n+1} \\
&= dM^n/n! + O_n(d^2M^{n-1}).
\end{align*}
So by using that $d^2/M = O_n(1)$, we have that $s^{1/n}/M = d^{1/n}/n!^{1/n} + O_n(d^2/M)$. Therefore the left-hand side can be replaced by
\[
\frac{d^{1/n}n}{n+1} \kappa - O_n\left(\frac{\ell d^2}{M}\right).
\]
For the right-hand side we have that 
\[
\frac{r-s}{Ms} 
= \frac{1}{d(n+1)} + O_n({1/M}),
\]
and so the inequality becomes
\[
\frac{d^{1/n}n}{n+1}\kappa - O_n\left(\frac{\ell d^2}{M}\right) \leq b-1 - \frac{h(f)}{d(n+1)} + O_n\left(\frac{h(f)}{M}\right).
\] 
Rearranging, we obtain that
\begin{equation}\label{eq:kappa.final.bound}
\kappa 
\leq \frac{(n+1)(b-1)}{nd^{1/n}} - \frac{h(f)}{nd^{1+1/n}} + O_n\left(\frac{h(f)}{d^{1/n}M}\right) + O_n\left(\frac{\ell d^{2-1/n}}{M}\right).
\end{equation}
We now take $M$ of size 
$O_n(d^{3}\ell+ d^{2}b)$, 
where the implicit constant will be determined later. We distinguish two cases depending on $h(f)$. Assume first that $h(f)\leq 2(n+1)db$. 
By taking the implicit constant for $M$ large enough, we may arrange for the last two terms in Equation~\eqref{eq:kappa.final.bound} to satisfy
\[
O_n\left(\frac{h(f)}{d^{1/n}M}\right) < \frac{1}{2nd^{1+1/n}}, \quad O_n\left(\frac{\ell d^{2-1/n}}{M}\right)< \frac{1}{2nd^{1+1/n}}.
\]
We then obtain that
\[
\kappa < \frac{(n+1)(b-1)}{nd^{1/n}} - \frac{h(f)-1}{nd^{1+1/n}},
\]
as desired. Second, assume that 
$h(f)\geq 2(n+1)db$. By adjusting the implicit constant for $M$, we then have in Equation~\eqref{eq:kappa.final.bound} that
\[
O_n\left(\frac{h(f)}{d^{1/n}M}\right) \leq \frac{h(f)}{2nd^{1+1/n}}, \quad O_n\left(\frac{\ell d^{2-1/n}}{M}\right) < 1/d,
\]
which leads to a contradiction as $\kappa$ is a positive number 
\[
\kappa 
\le 
\frac{(n+1)(b-1)}{nd^{1/n}} - \frac{2(n+1)db}{nd^{1+1/n}} + \frac{(n+1)db}{nd^{1+1/n}} + \frac{1}{d}
< 0.
\]
In conclusion, we obtain the desired contradiction by taking
\[
\kappa \geq \frac{(n+1)(b-1)}{nd^{1/n}} - \frac{h(f)-1}{nd^{1+1/n}}
\]
and $M$ of size $O_n(d^{1+1/n}b + d^3\ell)$.

%

\end{proof}

Using a technique due to Ellenberg--Venkatesh~\cite{Ellenb-Venkatesh}, we obtain the following affine variant. Note that the proof in this case is much simpler than in e.g.\ ~\cite[Prop.\,4.2.1]{CCDN-dgc}, by the ubiquity of height one primes in $\cO_K$.

Let $ i\ge 0$. 
For $f\in \cO_K[x_1, \ldots, x_{n+1}]$, we denote by $f_i\in \cO_K[x_1,\ldots, x_{n+1}]$ its degree $i$ homogeneous part, which is a homogeneous polynomial.

\begin{proposition}[Affine auxiliary polynomials]\label{prop:aff.aux.poly}
Let $f\in \cO_K[x_1, \ldots, x_{n+1}]$ be primitive, absolutely irreducible of degree $d$, and let $X\subset \AA^{n+1}_K$ be the affine variety cut out by $f$. Let $b$ be a positive integer and take positive integers $\mu_1, \ldots, \mu_\ell$ for which
\[
\sum_{i=1}^\ell \frac{1}{\mu_i^{1/n}}\geq  \frac{b-1}{d^{1/n}} - \frac{h(f_d)-1}{nd^{1+1/n}}.
\]
Let $p_1, \ldots, p_\ell$ be primes. 
Then if $P=(P_1,\ldots,P_\ell)\in \prod X_{p_i}(\CC)$ is a tuple where each $P_i$ is of multiplicity $\mu_i$, then there exists a polynomial 
\[
g:=g_{P}\in \cO_K[x_1, \ldots, x_{n+1}]
\]
coprime to $f$, of degree $O_n(d^{1+1/n}b + d^3\ell )$, and vanishing on
\[
X(b; P_1, \ldots, P_{\ell}) = \{x\in X(b): x\equiv P_i \bmod p_i\}.
\]
\end{proposition}

\begin{proof}
By Corollary~\ref{cor: existence of C points} there exists a point $a\in \CC^{n+1}$ which does not lie in $X(1)$.
By considering $f(x+a)$, we may assume that $f_0$ is non-zero.
Take an element $H\in \cO_K$ of height $b-1$  which does not divide $f_0$. Consider the primitive homogeneous polynomial 
\[
F(x_0, x_1, \ldots, x_{n+1}) = \sum_{i=0}^d H^i f_i x_0^{d-i}.
\]
We note that $F$ is absolutely irreducible because $x_0\nmid F$ and $F(H, x_1, \ldots, x_{n+1}) = H^df(x_1, \ldots, x_{n+1})$, and that $h(F)\geq d(b-1) + h(f_d)$. Every point $(x_1, \ldots, x_{n+1})$ on $X$ gives a point $(H : x_1 : \ldots : x_{n+1})$ on $V(F)$. By applying Theorem~\ref{thm:proj.aux.poly} to $F$ we get the desired result. 
\end{proof}

\subsection{Counting on curves}

Using Theorem~\ref{thm:proj.aux.poly} and Proposition~\ref{prop:aff.aux.poly} we can quickly deduce results on counting on projective and affine  curves. 

%

\begin{proof}[Proof of Theorem~\ref{thm: projective curves}]
By Lemma~\ref{lem: projection for proj varieties} we may assume that $n=2$, i.e.\ $C$ is a planar curve.
We may moreover assume that $C$ is geometrically integral, by Lemma~\ref{lem:irre.vs.abs.irre}. 
Put $\ell = \left\lfloor \frac{2(b-1)}{d}\right\rfloor +1$. 
Let $Z_1, \ldots, Z_k$ be the irreducible components of $C(b)$ of maximal dimension.
If $Z_j$ is completely contained in the singular locus  $C^{\mathrm{sing}}(\CC(t))$, then $Z_j$ is a finite set 
 and hence we are done. 
Otherwise we fix on every $Z_j$ a point $Q_j$ which is smooth as a point on $C$. 
Since being smooth and being geometrically integral are generic conditions, there exist non-zero primes $p_1, \ldots, p_\ell\in \cO_K$ such that each reduction $C_{p_i}$ is geometrically integral,  and such that for every $j$ the point $Q_j$ reduces to a smooth point on every $C_{p_i}$.
Let $Y\subset C(b)$ be the subset of points $P\in C(b)$ which reduce to a smooth point on every $C_{p_i}$, and note that $Q_j\in Y$ for every $j$.
Then combining Theorem~\ref{thm:proj.aux.poly} with all $\mu_i = 1$ and B\'ezout's theorem shows that the morphism
\[
Y\to \prod_i C_{p_i}
\]
is finite-to-one onto its image, and has degree bounded by $O(d^3b)$. 
Now, for each irreducible component $Z_j$ the set of $Q\in Z_j$ which reduce to a smooth point on every $C_{p_i}$ form an open subset of $Z_j$.
It follows that $Y$ is dense in every $Z_j$.
The result now follows by noting that $Y$ has at most $O(d^3b)$ irreducible components of dimension $\ell = \dim(\prod_i C_{p_i})$.
\end{proof}


\begin{proof}[Proof of Theorem~\ref{thm: affine curves}]
The argument is identical to the previous corollary, but using Proposition~\ref{prop: projection for affine varieties} instead of Lemma~\ref{lem: projection for proj varieties}, and Proposition~\ref{prop:aff.aux.poly} instead of Theorem~\ref{thm:proj.aux.poly}.
\end{proof}

\section{Geometric dimension growth}\label{sec:geom.dim.growth}

In this section we prove our main results on dimension growth. The projective case will be deduced from the affine case, which we will prove by induction on dimension. The base case for Theorem~\ref{thm:main.dgc} is counting on affine surfaces.  

For the affine case, we will prove a dimension growth result which is more general than what is needed for Theorem~\ref{thm:main.dgc}, similar to~\cite{Vermeulen:affinedg, CDHNV}.  

We embed $\AA^n_K$ in $\PP^n_K$ where the hyperplane at infinity $H_\infty$ is given by $x_0 = 0$. 
If $X\subset \AA^n_K$ is an affine variety, we denote by $\overline{X}$ the Zariski closure of $X$ in $\PP^n$.
Let $X_\infty = \overline{X}\cap H_\infty$, which is a projective variety in $\PP^{n-1}_K$.
If $X$ is defined by a degree $d$ polynomial $f\in K[x_1, \ldots, x_n]$, then $X_\infty$ is defined by $f_d = 0$ in $\PP^{n-1}_K$.

\begin{theorem}[Dimension growth for affine varieties]\label{thm:aff.dim.growth}
Let $X\subset \AA^n_K$ be an irreducible affine variety of dimension $m$ and of degree $d\geq 6$.   
Assume that $X_\infty$ has at least one geometrically irreducible component of degree at least $3$. Then for all $b\geq 1$
\[
\Naff(X,b)\leq (m-1)b + 1,
\]
and the number of irreducible components of $X(b)$ of dimension $(m-1)b+1$ is at most $O(d^7)$.
For $b\geq 7$ this number of irreducible components is at most  $O(d^4)$.
\end{theorem}

\begin{remark}\label{rem:NCC}
Note that the condition that $X$ is not cylindrical over a curve (abbreviated NCC) appearing in the main affine results of \cite{Vermeulen:affinedg,CDHNV}
holds if $X_\infty$ has at least one geometrically irreducible component of degree at least $3$ as above. Indeed, a hypersurface $X= V(f)$ is NCC if it is not the preimage of a planar curve under a $K$-linear map. 
 The condition on $X_\infty$ guarantees $f_d$ is not a form in two variables after any $K$-linear change of coordinates. 
\end{remark}
%

\subsection{Surfaces}

We first count on affine surfaces. The goal of this section is to prove the following proposition:

\begin{proposition}\label{prop:aff.surfaces}
Let $f\in \cO_K[x,y,z]$ be an absolutely irreducible polynomial of degree $d\geq 6$, and let $X = V(f)\subset \AA^3_K$ be the corresponding surface. Assume that 
 $f_d$ has at least one absolutely irreducible factor of degree at least $3$. Then for $b\geq 1$ we have
\[
N_{\mathrm{aff}}(X,b) \leq b+1.
\]
For $b\geq 7$, the number of irreducible components of $X(b)$ of dimension $b+1$ is at most $O(d^4)$, while for $b<7$ this number is $O(d^7)$.

Moreover, if $f_d$ is absolutely irreducible and $U\subset X$ is the open subset obtained by removing all lines from $X$, then
\[
N_{\mathrm{aff}}(U,b)\leq \max\left\{\frac{b+1}{2}, \left(\frac{3}{2\sqrt{d}} + \frac{1}{3}\right)(b-1)+\frac{5}{2}\right\}.
\]
\end{proposition}

For the proof, we will need a result on rulings of surfaces by curves of degree $1$ and $2$.
First of all let us state what we mean by a ruling.
By a \emph{conic} in $\AA^3_K$ we mean a curve of degree $2$.
Note that every conic must lie on a plane.

\begin{definition}
	Let $X \subset \PP_K^3$ be a projective surface.
	\begin{enumerate}
	\item We say that $X$ is \emph{ruled} if there exist infinitely many distinct lines contained in $X$.
	\item We say that $X$ is \emph{ruled tangentially by conics} if there exist infinitely many distinct conics contained in $X$ which are tangent to $H_\infty$.
\end{enumerate}	
In the above context, by referring to the  \textit{ruling} $\mathcal{R}$, we mean the respective sets of lines and conics satisfying the given conditions.
\end{definition}

Note that in the above definition the conics may be degenerate, i.e.\ they can be a double line. 

\begin{lemma}\label{lem:ruled}
	Let $X \subset \PP_K^3$ be a projective, geometrically integral surface of degree $d\geq 2$ and suppose 
	that $X$ is ruled or tangentially ruled by conics with a ruling $\mathcal{R}$.
	Assume that $X_\infty$ has at least one geometrically irreducible component $E$ of degree at least $3$.
	Then the following hold. 
	\begin{enumerate}
	\item \label{it:E.only} $E$ is the unique geometrically irreducible component of $X_\infty$ of degree at least $3$. 
	\item \label{it:every.x} For every $x\in X$ there exists a $C_x\in \cR$ containing $x$ which intersects $X_\infty$ in $E$.
	\item \label{it:not.cone} For every $e\in E$ there are only finitely many $C \in \mathcal{R}$ 
	passing through $e$.
	\item \label{it:fin.not.ruling} There are only finitely many 
$C \in \mathcal{R}$ 	 which do not intersect $E$.
\end{enumerate}		
\end{lemma}

\begin{proof}
	First of all we construct a projective $K$-variety $Z$ parameterizing conics in $\PP^3$.
	Note that every conic in $\PP^3$ is the intersection of a quadratic form and a hyperplane.
	Hyperplanes in $\PP^3$ are naturally parameterized by the dual space $(\PP^3)^* \cong \PP^3$, say with coordinates $[q_0:q_1:q_2:q_3]$.
	Consider an affine chart $q_i \neq 0$  in $(\PP^3)^*$ for some $ 0 \le i \le 3$.
	We can then write $x_i=-\sum_{j \neq i} \frac{q_j}{q_i} x_j$, so that the intersection of a quadratic form with this hyperplane will be entirely determined by a quadratic form $Q$ in the $x_j$, $j\neq i$. 
	The space of such quadratic forms is isomorphic to $\PP^5$.
	We then get a scheme $Z_i \cong \AA^3 \times \PP^5$.
	We construct $Z$ by gluing together the $X_i$ in the obvious way.
	
	Note that $Z$ comes equipped with a natural map $Z \to \PP^3$, projecting to the hyperplane containing the conic, and that $Z$ is a $\PP^5$-bundle over $\PP^3$, implying that $Z$ is complete.  
	Note also that conics which are in fact double lines are not uniquely parameterized by $Z$, since they are contained in infinitely many hyperplanes.
	
	Next, consider the subvariety $Z_\infty$ of $Z$ consisting of those conics tangent to $H_\infty$. 
	This is clearly a closed condition on $Z$. 
	Let $Y\subset Z_\infty$ be the  subvariety consisting of conics which are contained in $X$. 
	We argue that this is a closed subvariety of $Z$.
	Consider 
	\[
	\cI = \{(x,\gamma)\in \PP^3\times Z_\infty:  x\in \gamma\},
	\]	
	which is clearly a closed subvariety of $\PP^3\times Z_\infty$, and denote by $\pi_1: \cI\to \PP^3$ and $\pi_2: \cI\to Z_\infty$ the coordinate projections.
	Since $\PP^3$ is complete, we see that $\pi_2(\pi_1^{-1}(X)) = Y$ is closed in $Z_\infty$, as desired.
	
	Whether we have a ruling $\mathcal{R}$ by lines or by conics, we must have a geometrically irreducible component $D$ of $X_\infty$ such that infinitely many different curves $C \in \mathcal{R}$ intersect $D$. 	 
	Consider the projective variety 
	\[
	T :=\{ (p, q, L) \in X \times D \times Y : p, q \in L \}.
	\]
	The image of the projection $\pi:T \to X$ must be a closed subset of $X$.
	By assumption, this image contains infinitely many distinct $C \in \mathcal{R}$. 
	As $X$ is geometrically integral, we therefore have that $\pi(T)=X$.
	Hence, for \emph{every} $p \in X$, there is a
curve $C \in \mathcal{R}$ 	
lying on $X$ passing through $D$.
	
	In particular, through every point of $E \setminus D$ there is either a line or conic contained in $X_\infty$ intersecting $D$.
	As $E$ has degree at least $3$, and is geometrically irreducible, this would mean that $X_\infty$ has infinitely many irreducible components, a contradiction.
	Hence $D=E$ for every irreducible component $E$ of $X_\infty$ of degree at least $3$, proving (\ref{it:E.only}), (\ref{it:every.x}) and (\ref{it:fin.not.ruling}).
	
	For (\ref{it:not.cone}), note that if this statement would be false, then there exists some $e\in E$ such that for every point $x\in X$ there exists a line (resp.\ conic) through $e$ and $x$ contained in $X$. 	But then a similar argument as above shows that $X_\infty$ has infinitely many irreducible components.
\end{proof}

We first study the contribution to counting points of bounded height from lines and conics on the surface. 
For lines, we have the following lemma.

\begin{lemma} \label{lem:lines.aff.surface}
Suppose $f \in \mathcal{O}_K[x, y, z]$ is absolutely irreducible of degree $d \ge 2$ and let $X = V(f)\subset \AA^3_K$.
Assume that $f_d$ has at least one absolutely irreducible factor of degree at least $3$. 
Let $b$ be a positive integer, and denote by $Y\subset X$ the union of all lines $\ell$ on $X$ 
such that $\ell(b)$ contains at least two distinct points of $X(b)$. Then $Y(b)$ is a constructible set over $\CC$, we have 
\[
\dim Y(b)\leq b+1,
\]
and the number of irreducible components of $Y(b)$ of dimension $b$ is $O(d^4)$. 
\end{lemma}

\begin{proof} 
If $L$ is a line in $\AA^3_K$ then it is of the form
\[
L=\{a+\lambda v: \lambda\in K\},
\]
for some $a \in \AA^3_K$ and $v = (v_1, v_2, v_3)\in \cO_K^3\setminus \{(0,0,0)\}$, where we may assume that the $v_i$ are coprime. This line is contained in $X$ if and only if  for every $\lambda\in K$ we have that $a+\lambda v\in X$. By Corollary~\ref{cor: existence of C points} applied to $L \cap X \subseteq L$, it is enough to require this condition for all $\lambda\in \CC$. 
If $L$ contains at least two points on $X(b)$, then we may assume that $a\in X(b)$ and that $h(v)< b$. Consider the set
\[
\{(a,v)\in X(b)\times (\AA^3_K \setminus\{(0,0,0)\})(b): a+\lambda v\in X(b)\,\, \forall \lambda\in \CC\}.
\]
By Chevalley's theorem, this is a constructible set over $\CC$, and its projection onto the first coordinate is exactly $Y(b)$ and hence constructible.

Now let $L$ be such a line in $Y$, with base-point $a\in X(b)$ and $v\in \cO_K^3$ primitive. Then the set of integral points on $L$ is given exactly by
\[
L(\cO_K) = \{a+\lambda v: \lambda\in \cO_K\}
\]
and hence $\Naff(L,b) = b-h(v)$ and $L(b)$ is irreducible. 

If $v$ is a direction of a line $L$ on $Y(b)$, then by expanding $f(a+\lambda v)$ into powers of $\lambda$, we see that $f_d(v) = 0$. As discussed in Remark~\ref{rem:NCC}, $X$ is NCC and so~\cite[Lem.\,4.7]{Vermeulen:affinedg} shows that there can be at most $d^2$ lines contained in $X$ which are parallel to a given line $L$ on $X$. Combining these facts, the strategy is to bound the integral points on lines by counting possible directions for lines, which come from rational points on the curve $X_\infty\subset \PP^2_K$ defined by $f_d = 0$.

Let $E$ be a geometrically irreducible component of $X_\infty$ which rules $X$ as in Lemma~\ref{lem:ruled} if it exists.
Otherwise put $E = \varnothing$.
Lemma~\ref{lem:ruled} shows that $E$ is unique and of degree $e\geq 3$, if it exists.

We first count on lines not coming from $E$.
Let $g\in K[x,y,z]$ be the homogeneous factor of $f_d$ obtained by removing $E$, i.e.\ the vanishing locus of $g$ is exactly the Zariski closure of $X_\infty\setminus E$.
Then Lemma~\ref{lem:ruled} shows that projecting the set
\[
\{(a,v)\in \AA^3\times \PP^2_K :  g(v) = 0 \wedge \forall  \lambda\in K \left(f(a+\lambda v) =0\, \right)\} 
\]
onto its first coordinate yields a finite union of lines on $X$. 
Expanding $f(a+\lambda v)$ as a polynomial in $\lambda$ gives polynomials of degree at most $d$ in $a,v$.
Therefore~\cite[Lem.\,4.6]{Vermeulen:affinedg} implies that this set consists of at most $d^4$ lines.

All other lines on $X$ must come from the ruling and intersect $E$.
Given an integer $k$ with $1\leq k\leq b$, denote by $Y_k\subset Y(b)$ the subvariety of points on a line $L$ whose direction $v$ satisfies $v\in E$ and $h(v)=k-1$. 
Recall that for each such line $L$ we have $\dim L(b) \le b-k+1$. 
We will show that $\dim Y_k\leq b$ and count the number of irreducible components of $Y_k$.
By Theorem~\ref{thm: projective curves}, 
as $e \ge 3$, 
we get that  
\[
\dim E(k) \le 
\left\lfloor \frac{2(k-1)}{e}\right\rfloor + 1. 
\]
In particular, for each $k \ge 1$, the dimension of $Y_k$ is at most
\[
\dim Y_k \le \dim E(k) + (b-k+1)
\le \left\lfloor \frac{2(k-1)}{e}\right\rfloor +
 b-k
+ 2 \le b+1.
\]
For $k\geq 2$ this shows that the dimension of $Y_k$ is at most $b$. 
We thus only have to consider $Y_1$ in counting irreducible components.

To count the number of irreducible components, note that $E(1)$ has at most $e$ irreducible components of dimension $1$, by Corollary~\ref{cor: projective geometric SZ}.
For a line $L$ on $X$ there are at most $d^2$ lines on $X$ parallel to it, and hence $Y_1$ has at most $ed^2$ irreducible components of dimension $b$.
Putting everything together, we obtain that the number of irreducible components of $Y(b)$ of dimension $b+1$ is bounded by 
\[
O\left(d^4 + d^2 e \right) \leq O(d^4).\qedhere
\]
\end{proof}

For conics on a surface we have the following. 

\begin{lemma}\label{lem:conics.aff.surface}
Suppose $f \in \mathcal{O}_K[x, y, z]$ is absolutely irreducible of degree $d \ge 3$ and let $X = V(f)\subset \AA^3_K$. Assume that $f_d$ has an absolutely irreducible factor of degree at least $3$. 

Let $b$ be a positive integer, and denote by $Y\subset X$ the union of all conics on $X$ which are tangent to the plane at infinity in a $K$-rational point. Then $Y(b)$ is a constructible set over $\CC$ with
\[
\dim Y(b)\leq \frac{2b+1}{3}.
\]
If $f_d$ is absolutely irreducible and $d \ge 4$ then one may improve this to
\[
\dim Y(b)\leq \frac{b+1}{2}.
\]
Moreover, if $C$ is a conic lying on $X$ and not tangent to the plane at infinity, then $C(b)$ is at most one-dimensional for all $b$.
\end{lemma}

\begin{proof}
It follows that 
$Y(b)$ is a variety over $\CC$ in the same way as it did for ruling by lines. We will follow the strategy of Browning--Heath-Brown--Salberger \cite{Brow-Heath-Salb} in Section 6.1 to prove the other claims. 
	
Let $C$ be any non-degenerate conic contained in $X$. If $C(b)$ is of dimension $0$, then we are done. Otherwise, $C$ is determined in $\PP^3$ by a linear form $L(x_0, x_1, x_2, x_3)$ and a quadratic form $q(x_0, x_1, x_2, x_3)$, both with coefficients in $\cO_K$.

If $C$ is not tangent to $H_\infty$, then the result follows by Proposition~\ref{prop:curve.siegel} and Proposition~\ref{prop: projection for affine varieties}.
So we may assume that $C$ is tangent to $H_\infty$.
As $q$ has coefficients in $\cO_K$, the conic $C$ meets $H_\infty$ at a point $y\in H_\infty$. 
We will show that $N_{\text{aff}}(C, b) \le \frac{b-h(y)}{2}$.
Our strategy will be to show that the height $b$ points on $C$ can be parameterized by finitely many quadratic polynomials of a certain form.

We may assume without loss of generality that $L(x) = \sum_i a_i x_i$ with $a_3\neq 0$ for some $a_i\in \cO_K$. Hence we have on $C$ that
\[
x_3=a_3^{-1}(a_0 x_0+a_1x_1+a_2 x_2).
\]
By substituting this in the quadratic form $q$ defining $C$, we get a quadratic form $Q(x_0, x_1, x_2)$ which, together with the equation for $x_3$ above, determines $C$.
As $C$ is tangent to $H_\infty$, we have that  $Q(0, x_1, x_2)=a(\alpha x_1+\beta x_2)^2$, where we can assume that $\alpha$ and $\beta$ have no common factors.
Altogether we get,
\[
Q(x_0, x_1, x_2)=a (\alpha x_1+\beta x_2)^2+ b x_0 x_1+c x_0 x_2+d x_0^2.
\]
We will reparameterize such that $\alpha x_1+\beta x_2$ becomes a single variable; we will do this in a $\cO_K$-invertible way.

As $\cO_K$ is a PID, we can find elements $\gamma$ and $\delta$ such that $\alpha \delta-\beta \gamma=1$.
Denote $y_0=x_0$, $y_1=\alpha x_1+\beta x_2$, $y_2=\gamma x_1+\delta x_2$.
In terms of $y_0$, $y_1$, $y_2$ the quadratic form then becomes
\[
Q(y_0, y_1, y_2)=a y_1^2+e y_0 y_1+g y_0 y_2+d y_0^2.
\]
Note that $g \neq 0$ as $C$ is non-degenerate.
Now, using $Q(y_0, y_1, y_2)=0$, we see that we need
\[
y_2=\frac{-ay_1^2-dy_0^2-e y_0 y_1}{g y_0}
\]
if $y_0 \neq 0$.
Because we care about the affine coordinates on $C$, we will be assuming $y_0=1$.
Putting $y_0=x_0=1$, we then find that we can calculate $y_2$ in terms of $y_1$; hence we can find $x_1$, $x_2$ and $x_3$ in terms of $y_1$.
In particular, we have quadratic polynomials $q_1, q_2, q_3$ over $K$ such that $x_1=q_1(y_1)$, $x_2=q_2(y_1)$, $x_3=q_3(y_1)$ is a parameterization of $C$.

Now, if $C$ contains no $\cO_K$-points of height at most $b$, we are done. Otherwise, let $\delta$ be an $\cO_K$-point on $C$ of height at most $b$, and take $y^*\in K$ such that $\delta = (q_1(y^*), q_2(y^*), q_3(y^*))$. Note that $y^* = \alpha \delta_1 + \beta \delta_2$, showing that in fact $y^* \in \cO_K$.
We can then change coordinates to $u=y-y^*$. 
Abusing notation by putting $x_i = q_i(u)$, we obtain $\delta_i = q_i(0)$, and hence we may assume that the $q_i$ have constant term of height at most $b$ in $\cO_K$.

Next, let $D$ be a common denominator for all the $q_i$.
Suppose
\[
q_i(u)=\delta_i+\frac 1D (B_i u+ C_i u^2).
\]
We want to find those $u \in \cO_K$ for which $q_i(u)$ is in $\cO_K$.
We will subdivide them according to their greatest common divisor with $D$, so fix a factor $\lambda$ of $D$ and fix some $u \in \cO_K$ for which $q_i(u) \in \cO_K$ and $\gcd(u, D)=\lambda$.

Write $u=\lambda z$, $D=\lambda D'$.
We get that
\[
q_i(u) - \delta_i = \frac{1}{D'}(B_i z+\lambda C_i z^2) \in \cO_K
\]
for all $i$.
By assumption, $z$ has no common factor with $D'$, so we have $D' \mid B_i+\lambda C_i z$.
Let $G_i$ be the greatest common divisor of $\lambda C_i$ and $D'$.
Then we obtain $G_i \mid B_i$. Denote $\lambda C_i=G_i C_i'$, $B_i=G_i B_i'$ and $D'=G_i D_i''$.
Then we find that $z \equiv -C_i'^{-1} B'_i \pmod {D_i''}$, which determines $z$ uniquely modulo the least common multiple $d'$ of the $D_i''$. Denote by $Z\in \cO_K$ the element of smallest degree such that $Z \equiv -C_i'^{-1} B'_i \pmod {D_i''}$.

We can then substitute $z=Z+d' v$, with $v\in \cO_K$, to get representations $q_i(v)$ (in terms of $v$) which are contained in $\cO_K$.

Because we can repeat this for every factor $\lambda$ of $D$ (which is each time the greatest common divisor with $u$), we find finitely many quadratic representations $(q_{1i}(v), q_{2i}(v), q_{3i}(v))$ whose $\cO_K$-images together cover every $\cO_K$-point of $C$.

We can almost conclude the proof of the claim now.
Focus on one representation
\[
q_i(v)=A_i+B_i v+C_i v^2.
\]
If it contains a height at most $b$ point in its $\cO_K$-image, we may assume that the $A_i$ have height at most $b$, otherwise we are done.

Suppose without loss of generality that $\deg C_1 \ge \deg C_2, \deg C_3$.
Write $B_1=B_1'+D_1 C_1$, with $\deg B_1'<\deg C_1$.
Then it is clear that in order to have $\deg q_1(V)<b$, we need $\deg v(v+D_1)<b-\deg C_1$.
If $\deg D_1< \frac{b-\deg C_1}{2}$, we see that the only requirement is $\deg v< \frac{b-\deg(C_1)}{2}$, which means that $v$ can take values in an irreducible space of dimension $\frac{b-\deg C_1}{2}$ over $\CC$.

If $\deg D_1 \ge \frac{b-\deg C_1}{2}$, we have to distinguish between two cases:
\begin{itemize}
	\item If $\deg v \neq \deg D_1$, then $\deg v <b-\deg C_1-\deg D_1 \le \frac{b-\deg C_1}{2}$.
	\item If $\deg v=\deg D_1$, then $\deg(v+D_1)<\frac{b-\deg C_1}{2}$.
\end{itemize}

In either case $v$ can only take values in an irreducible space of dimension at most $\frac{b-\deg C_1}{2}$ over $\CC$.
We are now done once we realize that $[0:C_1:C_2:C_3]=y$, and hence $h(y) \le \deg C_1$.

Suppose now that there are only finitely many non-degenerate conics lying on $X$ which are tangent to $H_\infty$.
By the above, in this case $\dim C(b) \le \frac{b}{2}$, so in particular $\dim Y(b) \le \frac{b}{2}$.

Finally, suppose that there are infinitely many non-degenerate conics lying on $X$ which are tangent to $H_\infty$.
By Lemma \ref{lem:ruled}, all but finitely many of these conics pass through $E$, which corresponds to the unique absolutely irreducible factor of $f_d$ of degree at least $3$.
We can ignore these finitely many conics and assume all conics we consider pass through $E$.

For $y \in E$, let $Y_y$ be the subvariety of $Y(b)$ consisting of those conics intersecting $H_\infty$ in $y$.
We know that $Y_y$ is the union of finitely many $C(b)$ by Lemma \ref{lem:ruled}.
Hence
$$\dim Y_y \le \frac{b-h(y)}{2}.$$
Using Theorem~\ref{thm: projective curves}, we conclude that
$$\dim Y(b) \le \max_{1 \le k \le b} \frac{b-k}{2}+\left \lfloor\frac{2(k-1)}{3}\right \rfloor+1 \le \frac{2b+1}{3}.$$
Note that the latter is strictly smaller than $b$ unless $b=1$.
If $f_d$ is absolutely irreducible and $d \ge 4$, then $E$ has degree $d$, so that
\[
\dim Y(b)\leq \max_{1\leq k\leq b} \frac{b-k}{2} + \left\lfloor \frac{2(k-1)}{d}\right\rfloor+1 \leq \frac{b+1}{2}.\qedhere
\]
\end{proof}


\begin{lemma}\label{lem:tangent.mult}
Let $X\subset \AA_K^3$ be a geometrically integral surface. Let $U\subset X$ be the subset 
 of smooth points $x \in X$ for which the intersection multiplicity of $X$ with the tangent space $T_xX$ at $x$ is at most $2$. Then $X\setminus U$ consists of $O(d^4)$ irreducible components of dimension one. 
\end{lemma}

\begin{proof}
This follows from~\cite[Lem.\,10, 11]{Brow-Heath-Salb}. Indeed, note that the universal polynomials $\Phi_i(a_e; x_0, x_1, x_2, x_3)$ constructed there are of degree at most $O(d^4)$ in $x$.
\end{proof}

We need one more elementary lemma for the proof of Proposition~\ref{prop:aff.surfaces}.

\begin{lemma}\label{lem:bound.ell.ell'}
Let $d\geq 6, e\geq 3$ and $b\geq 1$ be integers.
Define $\ell = \lfloor (b-1)/\sqrt{d}\rfloor + 1$ and $\ell' = \lfloor (b-1)/e - \ell / (e-1)\rfloor+1$.
Then
\[
2\ell + \ell' \leq b+1.
\]
Moreover, if $b\geq 7$ then $2\ell + \ell' \le b$.
\end{lemma}

\begin{proof}
Approximating $\ell$ and $\ell'$ by similar expressions without $\lfloor\cdot\rfloor$ we see that
\begin{align*}
2\ell + \ell' \leq \frac{2(b-1)}{\sqrt{d}} + \frac{b-1}{e} + 3 - \frac{b-1}{\sqrt{d}(e-1)} - \varepsilon,
\end{align*}
for some small $0 < \varepsilon < 1/(\sqrt{d}(e-1))$.
If $b\geq 20$, then we see that this quantity is indeed bounded as stated.
So we may assume that $b < 20$.
Also, if $b=1$ or $b=2$ then a straightforward computation gives the desired conclusion, so we can assume that $3\leq b\leq 19$.
Now, if $b-1 < \sqrt{d}$ then $\ell = 1$, and so we also obtain the desired conclusion.
So we may assume that $\sqrt{d} \leq b-1 \leq 18$.
For fixed $b,d$ with $d\geq 6$ we note that $\ell'$ goes to $1$ as $e$ goes to $\infty$.
Hence for each $b,d$ there are only finitely many values of $e$ to check.
Going through all finitely many remaining cases using a computer program yields the proof.
\end{proof}

We have all of the ingredients for Proposition~\ref{prop:aff.surfaces}.

\begin{proof}[Proof of Proposition~\ref{prop:aff.surfaces}]
By Lemma~\ref{lem:irre.vs.abs.irre} we may assume that $X$ is geometrically integral.
Let $Y_1\subset X(b)$ consist of all integral points of height at most $b$ on the union of all lines on $X$ which contain at least two points of $X(b)$. Similarly, let $Y_2\subset X(b)$ consist of all integral points of height at most $b$ on $X$ on conics on $X$ which are tangent to the plane at infinity. By Lemmas~\ref{lem:lines.aff.surface} and  \ref{lem:conics.aff.surface} we have 
\[
\dim Y_1\leq b+1, \dim Y_2\leq \frac{2b+1}{3},
\]
and that $Y_1$ has at most $O(d^4)$ irreducible components of dimension $b+1$.
Let $Z\subset X$ consist of all points $x$ on $X$ which are either singular on $X$, or for which the tangent plane section has intersection multiplicity strictly greater than $2$ at $x$. Then Lemma~\ref{lem:tangent.mult} and the Schwartz--Zippel bound from Proposition~\ref{prop: affine geometric SZ} show that
\[
\Naff(Z,b)\leq b,
\]
and the number of irreducible components of $Z(b)$ of dimension $b$ is bounded by $O(d^4)$.

Let $(Z_j)_j$ be the irreducible components of $X(b)$ of maximal dimension. 
If $Z_j$ is contained in $Z(K)$ then $Z_j$ is in fact contained in $Z(b)$. Hence by the reasoning above we have $\dim Z_j\leq b$. 
Otherwise, we take a point $Q_j$ on $Z_j$ which is not in $Z$, i.e.\ it is smooth on $X$ and has intersection multiplicity at most $2$ with the tangent plane section. Put $\ell = \lfloor (b-1)/\sqrt{d}\rfloor+1$, which will be used to apply the determinant method later. Then there exist non-zero primes $p_1, \ldots, p_\ell\in \cO_K$ such that $X_{p_i}$ is geometrically integral, such that the reduction of $Q_j$ is still smooth on $X_{p_i}$, and such that the reduction still has intersection multiplicity at most $2$ with the tangent plane section at $Q_j\bmod p_i$. 
Indeed, all of these are generic conditions.

Let $Y\subset X(b)$ be the subset of points $P\in X(b)$ which reduce to a point which is smooth and has intersection multiplicity at most $2$ with the tangent plane section in every $X_{p_i}$. Consider the morphism
\[
\varphi: Y\to \prod_{i=1}^\ell X_{p_i}
\]
arising from reductions modulo $p_i$, $1 \le i \le \ell$. Fix a point $P = (P_1, \ldots, P_\ell)$ in the image of this map. By Proposition~\ref{prop:aff.aux.poly}, the preimage $\varphi^{-1}(P)\cap Y$ is contained in a set of the form $C(b)$ for a curve $C\subset X$ of degree at most $O(d^{7/2}b)$. In fact, $\varphi^{-1}(P)\cap Y$ is clearly even contained in $C(b; P_1, \ldots, P_\ell)$, consisting of those points $P\in C(b)$ such that $P\equiv P_i\bmod p_i$. Our aim is to bound the dimension of this object.

Let $D$ be an irreducible component of $C$, say of degree $e$. We may assume that $D$ is geometrically integral by Lemma~\ref{lem:irre.vs.abs.irre}. We will bound the dimension of $D(b; P_1, \ldots, P_\ell)$. If $D(b; P_1, \ldots, P_\ell)$ contains at most one point, then there is nothing to prove. So assume that $D(b; P_1, \ldots, P_\ell)$ consists of at least two points. If $e=1$ then $D(b; P_1, \ldots, P_\ell)$ is contained in $Y_1$. If $e=2$ then $D(b; P_1, \ldots, P_\ell)$ is either contained in $Y_2$, or the final statement of Lemma~\ref{lem:conics.aff.surface} applies. We may disregard either case and assume that $e\geq 3$.

Note that $1\leq \mult_{P_i}(D_{p_i}) \leq  e$. If this multiplicity were equal to $e$, then $D_{p_i}$ would be a cone, and hence $D_{p_i}$ would be contained in the intersection of $X_{p_i}$ with the tangent plane $T_{P_i}X_{p_i}$. But this contradicts the fact that $P_i$ has intersection multiplicity at most $2< e$ on the tangent plane section. Hence $\mult_{P_i}(D_{p_i}) < e$ for every $i$. Define the integer $\ell'$ as
\begin{equation}
\label{eq: ellprime}
\ell' = \ell'(e) = \left\lfloor \frac{b-1}{e} - \frac{\ell}{e-1}\right\rfloor +1.
\end{equation}
We now claim that 
\[
\dim D(b; P_1, \ldots, P_\ell)\leq \ell'.
\]
To prove the claim, we reason as in the proof of Theorem~\ref{thm: projective curves}. Let $(D_j)_j$ be the irreducible components of $D(b; P_1, \ldots, P_\ell)$ of maximal dimension. If $D_j$ is contained in $D^{\mathrm{sing}}(K)$, then $D_j$ is finite, and this certainly satisfies the desired dimension bound. Otherwise, fix a point $Q_j'$ on $D_j$ which is smooth on $D$. Take non-zero primes $q_1, \ldots, q_{\ell'}\in \cO_K$ for which $Q_j'$ remains smooth on the reduction $D_{q_j}$, and for which $D_{q_j}$ is geometrically integral. Let $D'\subset D(b; P_1, \ldots, P_\ell)$ consist of all points on $D(b; P_1, \ldots, P_\ell)$ which reduce to a smooth point on every $D_{q_j}$. Then Proposition~\ref{prop:aff.aux.poly} shows that the map
\[
D'\to \prod_{j=1}^{\ell'} D_{q_j}
\]
is finite-to-one onto its image, of degree at most $O(e^2b + e^3\ell')$. 
Indeed, let $Q_1'', \ldots, Q_{\ell'}''$ be smooth points on $D_{q_1}, \ldots, D_{q_\ell'}$. 
Since
\[
\ell' + \sum_{i=1}^\ell \frac{1}{\mult_{P_i}(D_{p_i})} \geq \frac{b-1}{e},
\]
Proposition~\ref{prop:aff.aux.poly} provides us with an auxiliary curve of degree at most $O(e^2b + e^3\ell')$ containing $D(b; P_1, \ldots, P_\ell, Q_1'', \ldots, Q_{\ell'}'')$ but not containing $D$.

As in the proof of Theorem~\ref{thm: projective curves}, we then obtain that $D'$ is dense in every $D_j$ of maximal dimension, and so we conclude that $D(b; P_1, \ldots, P_\ell)$ has dimension at most $\ell'=\ell'(e)$ as in Equation (\ref{eq: ellprime}). 

To conclude the proof, we have shown that every irreducible component of a fibre of the map $\varphi$ which is not contained in $Y_1$ or $Y_2$ has dimension bounded by 
\[
\max_{e\geq 3}\ell'(e).
\]
Now $\prod_{i=1}^{\ell} X_{p_i}$ has dimension $2\ell$, and so the union of all irreducible components of fibres of $\varphi$ which are not contained in $Y_1$ or $Y_2$ has dimension bounded by
\[
2\ell +  \max_{e\geq 3}\ell'(e).
\]
By Lemma~\ref{lem:bound.ell.ell'} this quantity is always bounded by $b+1$.
So we conclude that $X(b)$ has dimension bounded by $b+1$.

For the number of irreducible components, Lemma~\ref{lem:bound.ell.ell'} shows that for $b\geq 7$ only irreducible components from $Y_1$ or $Y_2$ can contribute.
For these, Lemma~\ref{lem:lines.aff.surface} show that this number is at most $O(d^4)$.
When $b < 7$, Proposition~\ref{prop:aff.aux.poly} shows that the map $\varphi: Y\to \prod_{i=1}^\ell X_{p_i}$ has degree at most $O(d^7)$, and so $X(b)$ has $O(d^7)$ irreducible components of dimension $b+1$.

For the moreover part of the statement, let $U\subset X$ be the open subset obtained by removing all lines from $X$.
Then in the above argument we do not have to consider $Y_1$.
For $Y_2$, we have by Lemma~\ref{lem:conics.aff.surface} that
\[
\dim Y_2(b) \leq \frac{b+1}{2}.
\]
The set $Z$ is a one-dimensional algebraic set.
Let $Z' \subset Z$ be obtained by removing all lines from $Z$, so that $U\cap Z \subset Z'$.
Then by Theorem~\ref{thm: affine curves} we have that
\[
\dim Z'(b)\leq \left\lceil \frac{b}{2}\right\rceil.
\]
Finally, we have argued above that all components of $X(b)$ which do not come from $Y_1, Y_2$ or $Z$ have dimension at most
\[
2\ell + \max_{e\geq 3}\ell'(e)\leq \left(\frac{3}{2\sqrt{d}} + \frac{1}{3}\right)(b-1)+\frac{5}{2}.
\]
Therefore we conclude that $U(b)$ has dimension at most
\[
\max\left\{\frac{b+1}{2}, \left(\frac{3}{2\sqrt{d}} + \frac{1}{3}\right)(b-1)+\frac{5}{2}\right\}. \qedhere
\]
\end{proof}

\subsection{Higher-dimensional varieties}

In this section we develop the necessary slicing arguments to deduce our main results. 

\begin{proposition}
\label{prop: Bertini}
Let $n\geq 3$ and let $X\subset \PP^n_K$ be a geometrically integral hypersurface.
Let $Z\subset (\PP^n_K)^*$ consist of all hyperplanes $H$ for which $X\cap H$ is not geometrically integral.
Then $\dim Z(1) < n$.
\end{proposition}

In particular, note that this implies that there are infinitely many hyperplanes $H\subset \PP^n$ defined over $\CC$ for which $X\cap H$ is geometrically integral.

\begin{proof}
By Bertini's Thereom, the variety of linear forms $\ell_\alpha$ over $K$ for which $f|_{\{\ell_\alpha=0\}}$ is reducible is a Zariski closed set strictly contained in the space of all linear forms on $\AA^n_K$. 
We are now done by Corollary \ref{cor: existence of C points}.
\end{proof}

\begin{lemma}\label{lem:Bertini.aff}
Let $n\geq 4$, let $X\subset \AA^n_K$ be a geometrically integral hypersurface and assume that $X_\infty$ has a geometrically irreducible component of degree at least $3$.
Let $Z\subset (\PP^n_K)^*$ consist of all hyperplanes $H$ for which $X\cap H$ is not geometrically integral, or for which $(X\cap H)_\infty$ has no geometrically irreducible component of degree at least $3$.
Then $\dim Z(1) < n$.
\end{lemma}

Again, this implies the existence of infinitely many hyperplanes $H\subset \AA^n$ defined over $\CC$ for which $X\cap H$ is geometrically integral and has a geometrically irreducible component of degree at least $3$ at infinity.

\begin{proof}
This follows by applying the previous proposition to $X$ and to the geometrically irreducible component of $X_\infty$ of degree at least $3$.
\end{proof}

\subsection{Proofs of main results}

In this section we finish up and prove our main results. 

\begin{lemma}\label{lem:induct.thm.A}
Let $n\geq 3$, let $f\in K[x_1, \ldots, x_n]$ be absolutely irreducible of degree $d$, and denote $X = V(f)\subset \AA^n_K$.
Assume that $f_d$ is not a power of a linear form.
Let $Z\subset (\PP^n_K)^*$ consist of all hyperplanes for which $(X\cap H)_\infty$ is the vanishing locus of a power of a linear form.
Then $\dim Z(1)< n$.
\end{lemma}

\begin{proof}
Let $D_1, \ldots, D_k\subset H_\infty=\PP^{n-1}_K$ be the geometrically irreducible components of $X_\infty$.
Then either $k > 1$ or $k=1$ and $D = D_1$ is not linear.

Assume first that $k>1$, so that $D_1\neq D_2$.
The set $Y\subset (\PP^{n-1}_K)^*$ consisting of all hyperplanes $H'\subset \PP^{n-1}_K$ for which $H'\cap D_1 = H'\cap D_2$ is a proper closed subvariety of $(\PP^{n-1}_K)^*$.
Hence so is the set $Z'\subset (\PP^n_K)^*$ of hyperplanes $H\subset \PP^n_K$ for which $H\cap H_\infty\in Y$.
But $Z$ is contained in $Z'$, so we conclude by Corollary~\ref{cor: projective geometric SZ}.

Now assume that $k = 1$ and write $D = D_1$.
Then $D$ is not linear.
If $n\geq 4$ then Bertini's theorem implies that the set $Y\subset (\PP^{n-1}_K)^*$ consisting of all hyperplanes $H'\subset \PP^{n-1}_K$ for which $H'\cap D$ is not geometrically integral is a proper closed subvariety of $(\PP^{n-1}_K)^*$.
But if $H\in Z$ then $H\cap H_\infty \in Y$ and so the result follows from Corollary~\ref{cor: projective geometric SZ}.
If $n = 3$ then $D$ is a curve in $\PP^2_K$ of degree at least $2$.
Hence, if $H'\subset \PP^2_K$ is a generic line, then $H'\cap D$ consists of $\deg D$ many isolated points.
These isolated points cannot be the vanishing locus of a linear polynomial, from which the result follows.
\end{proof}

For Theorem~\ref{thm:main.siegel} we require the following.

\begin{theorem}\label{thm:aff.dgc.Siegel}
Let $X\subset \AA^n_K$ be an irreducible hypersurface defined by $X=V(f)$ for $f\in K[x_1, \ldots, x_n]$. Assume that $f$ is of degree $d\geq 2$, and that $f_d$ is not a power of a linear form. Then for every $b \ge 1$,
\[
N_{\aff}(X,b)\leq (n-2)b+1.
\]
\end{theorem} 

\begin{proof}
If $X$ is irreducible but not geometrically irreducible then the result follows from Lemma~\ref{lem:irre.vs.abs.irre}. So we may assume that $X$ is geometrically irreducible. 

We induct on $n$, where the case $n=2$ follows from Proposition~\ref{prop:curve.siegel}. So assume that $n\geq 3$. By Lemma~\ref{lem:induct.thm.A}, there exists a hyperplane $H\subset \AA^n_K$ defined over $\CC$ such that $X\cap H$ is geometrically integral, and such that $(X\cap H)_\infty$ is not the vanishing locus of a power of a linear form. Without loss of generality, we may assume $H$ is defined by 
\[
x_n = a_1 x_1 + \ldots a_{n-1}x_{n-1}.
\]
For $c\in K$, denote by $H_c$ the hyperplane defined by $x_n = a_1x_1 + \ldots a_{n-1}x_{n-1}+c$. There are only finitely many $c\in \AA^1_K$ for which $X\cap H_c$ is either not geometrically irreducible, or $(X\cap H_c)_\infty$ is a power of a linear form. For these values of $c$, we use the Schwartz--Zippel bounds to obtain $N_{\aff}(X\cap H_c, b)\leq (n-2)b$. For the other values of $c$, by induction we have that $N_{\aff}(X \cap H_c, b)\leq (n-3)b+1$. Since $X(b)$ is the union of the $(X\cap H_c)(b)$ for $c$ running over $\AA^1_K(b)$, we have that $X(b)$ is of dimension at most $(n-2)b + 1$.
\end{proof}

\begin{proof}[Proof of Theorem~\ref{thm:main.siegel}]
Let $X\subset \PP^n_K$ be an irreducible projective hypersurface of degree $d\geq 2$. Let $Y\subset \AA^{n+1}_K$ be the affine cone over $X$. Then the conditions from Theorem~\ref{thm:aff.dgc.Siegel} are satisfied, and so we immediately obtain that
\[
N(X,b)\leq N_{\aff}(Y,b)-1\leq (n-1)b. \qedhere
\]
\end{proof}

\begin{proof}[Proof of Theorem~\ref{thm:aff.dim.growth}]
Let $X\subset \AA^n_K$ be an irreducible affine variety of degree $d\geq 6$ and of dimension $m\geq 2$. Assume that $X_\infty$ has a geometrically irreducible component of degree at least $3$.
By Lemma~\ref{lem:irre.vs.abs.irre} and the Schwartz--Zippel bound, we may assume that $X$ is geometrically integral. Using Lemma~\ref{prop: projection for affine varieties} we may assume that $X$ is a hypersurface in $\AA^n_K$. 

We now induct on $n$, where the base case $n=3$ is Proposition~\ref{prop:aff.surfaces}. So assume that $n\geq 4$. By Lemma~\ref{lem:Bertini.aff} there exists a hypersurface $H$ defined over $\CC$ such that $X\cap H$ is geometrically integral and such that $(X\cap H)_\infty$ has a geometrically irreducible component of degree at least $3$. 
Without loss of generality, $H$ is given by $x_n = a_1x_1 + \ldots + a_{n-1}x_{n-1}+a_n$ for some $a_i\in \CC$. 
For $c\in \cO_K$ let $H_c$ be the hyperplane given by $x_n = c+ \sum_{i=1}^{n-1}a_i x_i$. 
Then the set of $c\in K$ for which $X\cap H_c$ is not geometrically integral, or for which $(X\cap H_c)_\infty$ has no geometrically irreducible component of degree at least $3$, is finite. 
Consider the map $g: \AA^n_K\to \AA^1_K$ defined by
\[
g(x)=x_n - \sum_{i=1}^{n-1} a_i x_i,
\] 
 and note that 
if $x\in H_c$ then $g(x)=c$, and that since $H$ is defined over $\CC$ we have $h(g(x)) \leq h(x)$. 
Now, the map  $f: X(b)\to \AA^1_K(b)$ induced by $g$ 
is algebraic, and by induction all but finitely many fibres have dimension bounded by $(m-2)b+1$. The remaining finitely many bad fibres have dimension bounded by $(m-1)b$ by the Schwartz--Zippel bound. So we conclude that the dimension of $X(b)$ is bounded by  $(m-1)b+1$.
Moreover, by induction the number of irreducible components of $X(b)$ of dimension $(m-1)b+1$ is at most $O(d^4)$ if $b\ge 7$ and $O(d^7)$ otherwise. 
\end{proof}

\begin{proof}[Proof of Theorem~\ref{thm:main.dgc}]
We may assume that $X$ is geometrically integral. Let $Y\subset \AA^{n+1}_K$ be the affine cone over $X$. 
Then $Y$ is geometrically integral since $X$ is. 
Moreover, $Y_\infty$ is equal to $X$, and in particular has a geometrically irreducible component of degree at least $3$. 
Applying Theorem~\ref{thm:aff.dim.growth} then gives that 
\[
N(X,b)= \Naff(Y,b)-1\leq mb. 
\]
The bound on the number of irreducible components of  dimension $mb$ similarly follows. 
\end{proof}

\begin{proof}[Proof of Theorem~\ref{thm:proj.surfaces}]
We may assume that $X$ is geometrically irreducible, and recall that $U$ is obtained from $X$ by removing all lines.
Let $Y\subset \AA^4_K$ be the affine cone over $U$, so that
\[
N(U,b)\leq N_{\mathrm{aff}}(Y,b)-1.
\]
Then $Y$ is a geometrically integral quasi-affine $3$-fold, and $Y_\infty$ is also geometrically integral.
By Lemma~\ref{lem:Bertini.aff}, there exists a hyperplane $H\subset \AA^4_\CC$ defined over $\CC$ for which $Y\cap H$ and $(Y\cap H)_\infty$ are both geometrically integral.
Let $H$ be given by $a_1x_1 + \ldots +a_nx_n = a_0$ for some $a_i, a_0\in \CC$.
For $c\in \CC(t)$ let $H_c\subset \AA^4_K$ be the hyperplane given by $a_1x_1+\ldots + a_nx_n = c$.
Now $(Y\cap H_c)_\infty = (Y\cap H)_\infty$ is geometrically integral, and so for every $c\in \CC(t)$ we have that $Y\cap H_c$ is geometrically integral.

Take any $c\in \CC(t)\setminus \{0\}$, so that $0$ does not lie in $H_c$.
If $\ell$ is a line on $H_c\cap Y$, then because $Y$ is a cone with cone point $0$ the plane containing $\ell$ through $0$ would be contained in $Y$.
But this would yield a line on $U$, which by construction cannot exist.
Hence $H_c\cap Y$ contains no lines and we conclude from Proposition~\ref{prop:aff.surfaces} that
\[
N_{\mathrm{aff}}(H_c\cap Y , b)\leq \max\left\{\frac{b+1}{2}, \left(\frac{3}{2\sqrt{d}} + \frac{1}{3}\right)(b-1)+\frac{5}{2}\right\}.
\]
For $H_0\cap Y$ we have by Proposition~\ref{prop:aff.surfaces} that $N_{\mathrm{aff}}(H_0\cap Y, b)\leq b+2$.
By fibering over $c\in \AA^1_K(b)$, we conclude as desired that
\[
N_{\mathrm{aff}}(Y, b)\leq \left(1+\max\left\{ \frac{1}{2}, \frac{3}{2\sqrt{d}} + \frac{1}{3}\right\}\right) b + 2. \qedhere
\]
\end{proof}

\bibliographystyle{amsalpha}
\bibliography{anbib}
\end{document}